\documentclass[reqno]{amsart}
\usepackage[top=80pt,bottom=80pt,left=80pt,right=80pt]{geometry}
\usepackage[utf8]{inputenc}
\usepackage[english]{babel}
\usepackage{amsmath,amssymb,amsfonts,mathrsfs,mathtools}
\mathtoolsset{showonlyrefs}
\usepackage[T1]{fontenc}

\usepackage{enumitem}
\usepackage{hyperref}

\theoremstyle{plain}
\newtheorem{theorem}{Theorem}[section]

\newtheorem{remark}[theorem]{Remark}
\newtheorem{corollary}[theorem]{Corollary}
\newtheorem{definition}[theorem]{Definition}
\newtheorem{lemma}[theorem]{Lemma}
\newtheorem{proposition}[theorem]{Proposition}

\newtheorem{question}{Question}

\newcommand{\C}{\mathbb{C}}

\newcommand{\N}{\mathbb{N}}

\newcommand{\R}{\mathbb{R}}

\DeclarePairedDelimiter\abs{\lvert}{\rvert}

\renewcommand{\epsilon}{\ensuremath\varepsilon}
\newcommand{\eps}{\epsilon}

\renewcommand{\phi}{\ensuremath{\varphi}}

\newcommand{\U}[3]{\tilde{U}_{#2}^{#1}(#3)}

\renewcommand{\u}[3]{U_{#2}^{#1}(#3)}

\newcommand{\Oh}[3]{\tilde{O}_{#2}^{#1}(#3)}
\newcommand{\oh}[3]{O_{#2}^{#1}(#3)}

\newcommand{\vh}[3]{V_{#2}^{#1}(#3)}

\newcommand{\T}[1]{ \delta_{T}^{#1}\Gamma}
\newcommand{\stab}[1]{ \delta_{f}^{#1}\Gamma}

\newcommand{\co}[1]{\delta_\ast^{#1}\Gamma}
\newcommand{\cm}{\co{}}

\renewcommand{\P}[2]{{#2}^{(#1)}}
\newcommand{\G}[2]{\mathcal{G}_e^{#1}{#2}}

\newcommand{\GG}[2]{\tilde{\mathcal{G}}_e^{#1}{#2}}

\newcommand{\M}[2]{\partial_\ast^{#1} #2}
\newcommand{\m}[1]{\partial_\ast #1}

\renewcommand{\C}[1]{[#1]}

\newcommand{\e}{\ensuremath{e}}

\newcommand{\into}{\hookrightarrow}

\author{Stefanie Zbinden}
\title{Morse boundaries of graphs of groups with finite edge groups}

\begin{document}

\maketitle

\begin{abstract}
In this paper we prove that the Morse boundary of a free product depends only on the Morse boundary of its factors. In fact, we also prove the analogous result for graphs of groups with finite edge groups and infinitely many ends. This generalizes earlier work of Martin and \'Swi\k{a}tkowski in the case of non-hyperbolic groups.
\end{abstract}

\section{Introduction}

The Morse boundary has been introduced in \cite{Cor16} building on \cite{Charney_2014} and is a generalisation of the Gromov boundary for non-hyperbolic spaces. It captures hyperbolic properties of non-hyperbolic spaces and groups and, similar to the Gromov boundary, is a quasi-isometry invariant and visibility space. 

Morse boundaries inherit many properties of the Gromov boundary (see e.g. \cite{CH17}, \cite{Charney_2019}, \cite{Murray_2019}, \cite{liu_dynamics}, \cite{Zalloum2018ASC}) and \cite{charney2020complete} gives a topological description of certain Morse boundaries. While in this paper we work with the topology introduced in \cite{Cor16}, Cashen-Mackay introduced a different topology of the Morse boundary in \cite{CM19}, which most notably makes the Morse boundary metrizable. A summary of properties of the Morse boundary can be found in \cite{cordes2017survey}. For more information about hyperbolic spaces, we recommend \cite{hyperbolic_spaces} and for information about Cayley graphs and free products we recommend \cite{serre2012trees} although we will not follow the notation used there. 

Motivated by Dunwoody's accessibility \cite{Dunwoody1985}, that is, the fact that a finitely presented group $G$ splits as a graph of groups with finite edge groups and one-ended vertex groups, we show that the Morse boundary of $G$ only depends on the Morse boundaries of the vertex groups. Thus, this paper reduces the study of Morse boundaries (at least of finitely presented groups) to that of one-ended groups. However, our main result below also applies to non-finitely-presented groups, and when vertex groups are not one-ended:

\begin{theorem}\label{thm:graph_of_groups} For $i=1, 2$, let $\mathcal{G}_i$ be graphs of groups where all edge groups are finite and suppose that the groups $\pi_1\mathcal{G}_i$ have infinitely many ends. Denote by $h(\mathcal{G}_i)$ the set of homeomorphism types of Morse boundaries of the vertex groups in $\mathcal{G}_i$ that are non-finite and not virtually cyclic. If $h(\mathcal{G}_1) = h(\mathcal{G}_2)$, then the Morse boundaries $\m{\pi_1(\mathcal{G}_i)}$ are homeomorphic.  
\end{theorem}

In combination with \cite[Theorem 1.5]{charney2020complete}, we note that Theorem \ref{thm:graph_of_groups} implies, for example, that all fundamental groups of 3-manifolds of the form $M_1\# M_2$, where the $M_i$ are cusped hyperbolic 3-manifolds, have homeomorphic Morse boundaries.

Our main result generalises the main result of \cite{martin2013hyperbolic}, proving the same result for the Gromov boundary in the case of hyperbolic groups. Theorem \ref{thm:graph_of_groups} follows directly  from Theorem \ref{freeproduct_theorem} by applying Theorem 0.3 of \cite{papasoglu2002quasi}.

\begin{theorem}\label{freeproduct_theorem}
Let $A_1, A_2$ and $B_1, B_2$ be finitely generated groups that are infinite. Suppose $\m A_1\cong \m A_2$ and $\m B_1 \cong \m B_2$. Then, $\m (A_1\ast B_1) \cong \m (A_2 \ast B_2)$.
\end{theorem}

\subsection*{Further Directions}
A natural question is if Theorem \ref{thm:graph_of_groups} and Theorem \ref{freeproduct_theorem} also hold for the topology of the Morse boundary as defined in \cite{CM19}. 
\begin{question}
Do Theorem \ref{thm:graph_of_groups} and Theorem \ref{freeproduct_theorem} hold for the topology defined in \cite{CM19}? 
\end{question}

The sublinearly Morse boundary was defined in \cite{QR20} and can be viewed as a generalization of the Morse boundary. Hence, a similar question can be asked.

\begin{question}
Do Theorem \ref{thm:graph_of_groups} and Theorem \ref{freeproduct_theorem} hold for the sublinearly Morse boundary?
\end{question}

\subsection*{Outline of proof} We follow the proof of \cite{martin2013hyperbolic} closely, adapting their ideas and definitions to non-hyperbolic groups. That is, we give a combinatorial description of the Morse boundary of a free product and then use the homeomorphism $p:\m{A_1}\to\m{A_2}$ to construct a bijection $\hat{p}: A_1\to A_2$. The bijection $\hat{p}$ we construct induces the homeomorphism $p$. So another consequence of our proof is Corollary \ref{cor:homeo_induced_by_bijection}, which states that every homeomorphism of the Morse boundary of two groups is induced by a bijection of the groups. In the construction of the bijection, we deviate from the proof of the hyperbolic case. In the hyperbolic case, the bijection $\hat{p}$ has to be compatible with the homeomorphism $p$ in the sense that any sequence $(x_i)$ in $A_1$ that converges to some boundary element $x$ in $\partial A_1$ satisfies that the sequence $(\hat{p}(x_i))$ converges to $p(x)$. However, in our case, the bijection $\hat{p}$ also has to satisfy that for every Morse gauge $M$, there exists a Morse gauge $M^\prime$ such that $\hat{p}(\P{M}{A_1})\subset\P{M^\prime}{A_2}$. The main challenge in constructing $\hat{p}$ is that the sets $\P{M}{A_1}$ do not necessarily have a total order. We managed to construct a bijection satisfying all conditions by picking a Morse geodesic line $\lambda$ and then defining $\hat{p}(a)$ as one of the points far along $p(a\cdot \lambda)$. With the correct technical conditions of far away enough (which can be obtained by mapping the points back to $A_1$), the bijection satisfies all the required conditions and induces a homeomorphism from $\m{(A_1\ast B_1)}$ to $\m{(A_2\ast B_2)}$ (see Section \ref{section:bijections_on_factors} Condition \ref{map:construction:back_mapping}). A more detailed outline of the proof of Theorem \ref{freeproduct_theorem} can be found in Section \ref{section:proof_outline}.

\subsection*{Outline of paper} In Section~\ref{chapter:morse} we recall definitions and properties around the Morse boundary and prove some technical lemmas about the Morse boundary. In Section~\ref{chapter:comb} we give an alternate description of the Morse boundary of a free product. In Section~\ref{chapter:map} we show that the Morse boundary of a free product depends only on the Morse boundary of its factors. Finally in Section~\ref{chapter:graph_of_groups} we prove Theorem~\ref{thm:graph_of_groups}. 

\subsection*{Acknowledgements} This paper was my Master Thesis. I want to thank my supervisors Alessandra Iozzi and Matthew Cordes for their helpful guidance throughout my work on the thesis. I also want to thank Alessandro Sisto for helping to find the topic of the thesis. Lastly, I want to thank the referee, whose comments have led to clarifications and helped improve the exposition. 

\section{Morse Properties}\label{chapter:morse}

\subsection{Morse geodesics}
\begin{definition}[Morse gauge] A Morse gauge is a map $M : \R_{\geq 1}\times \R_{\geq_0}\to \R_{\geq0}$. We denote the set of all Morse gauges by $\mathcal{M}$.
\end{definition}

In this section unless stated otherwise, $(X,e)$ is a pointed proper metric geodesic space and $M$ a Morse gauge. For points $p, q\in X$ we denote by $[p, q]$ any geodesic from $p$ to $q$.

The set of all Morse gauges $\mathcal{M}$ has a partial order defined by $M\leq M^\prime$ if and only if $M(\lambda, \eps)\leq M^\prime(\lambda, \eps)$ for all $(\lambda, \eps) \in \R_{\geq 1}\times \R_{\geq 0}$. Observe that with this order, the maximum of a set of Morse gauges is the pointwise maximum.

\begin{definition}[Morse geodesic]
A geodesic $\gamma$ in $X$ is called $M$-Morse, if for all $(\lambda, \eps)$-quasi-geodesics $\eta$ with endpoints on $\gamma$ we have
\begin{align}
    \eta \subset \mathcal{N}_{M(\lambda, \eps)}(\gamma),
\end{align}
where $\mathcal{N}_{M(\lambda, \eps)}(\gamma)$ denotes the closed $M(\lambda, \eps)$-neighbourhood of $\gamma$ in $X$.
A geodesic is called Morse, if it is $M$-Morse for some Morse gauge $M$.
\end{definition}

\begin{remark} \label{morse:finite}
Observe that every geodesic segment $\gamma$ is $M$-Morse, where $M$ depends only on the length of $\gamma$. Namely, a $(\lambda, \eps)$-quasi-geodesic with endpoints that are at distance at most $\abs{\gamma}$ has bounded length $n$, depending only on $\lambda, \eps$ and $\abs{\gamma}$. Consequently, the quasi-geodesic has to stay in a $C$ neighbourhood of $\gamma$ where $C$ depends only on $\lambda, \eps$ and $\abs{\gamma}$.
\end{remark}

\begin{definition}[$\delta_M$] Let $M$ be a Morse gauge. We define
\begin{align}
    \delta_M = \max\{4 M (1, 2M(5, 0) ) + 2 M (5, 0), 8 M (3, 0) \}.
\end{align}
\end{definition}

The value $\delta_M$ is a measure of how close a point or geodesic is from some $M$-Morse geodesic $\gamma$, where all distances smaller than $\delta_M$ are considered to be close.

Morse geodesics have similar properties as geodesics in hyperbolic spaces. We cite some of these properties proven in \cite{Charney_2014} and \cite{Cor16}. Additionally, we prove some more results. 

\begin{lemma}[Adapted version of Lemma 2.5 of \cite{Charney_2014}]\label{lem_2.5_3}
Let $\gamma$ be an $M$-Morse quasi-geodesic in a geodesic metric space $X$.
\begin{enumerate}[label = \roman*)]
    \item If $\rho$ is a quasi-geodesic whose Hausdorff distance from $\gamma$ is at most $C$, then $\rho$ is $M^\prime$-Morse where $M^\prime$ depends only on $M$ and $C$.
    \item If $Y$ is a geodesic metric space and 
    $f : X \to Y$ is a $(\lambda, \eps)$-quasi-isometry, then $f \circ \gamma$ is
    $M^{\prime\prime}$-Morse where $M^{\prime\prime}$ depends only on $\lambda, \eps$ and $M$.
    \item If $\gamma$ is a $(K, L)$-quasi-geodesic, then there exists $C$ depending only on $M, K, L$ such that for any two points $x = \gamma(t)$ and $y = \gamma(t^\prime)$ on $\gamma$,
    every choice of geodesic $[x, y]$ has Hausdorff distance at most $C$ from $\gamma([t, t^\prime])$.
\end{enumerate}
\end{lemma}

The original version of Lemma~2.5 of \cite{Charney_2014} assumes that $X$ is a CAT(0) space. However, as noted in \cite{Cor16}, the statement also holds if $X$ is a geodesic metric space.

\begin{corollary}[Cor 2.5 and 2.6 of \cite{Cor16} combined]\label{cor2.5_2.6_cor16}
Let X be a geodesic metric space and $\alpha : [0, \infty) \to X$ be an $N$-Morse geodesic ray. 
\begin{enumerate}[label= \roman*)]
    \item Let $\beta: [0, \infty)\to X$ be a geodesic ray such that $\beta(0) = \alpha(0)$ and $d(\alpha(t), \beta(t)) < K $ for $t\in [0, D]$ for some $D\geq 6K$. Then $d((\alpha(t), \beta(t)) < \delta_N$ for all $t\in [0, D - 2K]$.
    \item Let $\beta: [0, \infty)\to X$ be a geodesic ray such that $d(\alpha(t), \beta(t))<K$ for all $t\in [0, \infty)$. Then for all $t\in [2K, \infty)$, $(\alpha(t), \beta(t)) < \delta_N + d(\alpha(0), \beta(0))$. In particular, if $\alpha(0) = \beta(0)$ then $d(\alpha(t), \beta(t))<\delta_N$ for all $t\in [0, \infty)$. 
\end{enumerate}
\end{corollary}

Corollary~\ref{cor2.5_2.6_cor16} ii) shows that two geodesic rays that have the same direction are close from some point on. Part i) shows that two geodesics (where at least one of them is Morse) that are close at the beginning and not too far away at the end are close for a long time. The next Lemma proves a similar but slightly different statement. 

\begin{lemma}[Lemma 2.7 of \cite{Cor16}]
Let $X$ be a geodesic metric space and let $\alpha_1, \alpha_2: [0, A]\to X$ be $N$-Morse geodesics with $\alpha_1(0) = \alpha_2(0)$. If $d(\alpha_1(s), \mathrm{im}(\alpha_2))<K$ for some $K>0$ and $s\in [0, A]$, then $d(\alpha_1(t), \alpha_2(t))\leq 8 N(3, 0)< \delta_N$ for $t< s-K-4 N(3, 0)$.
\end{lemma}

The following lemma from \cite{Charney_2019} shows that in a triangle, if two sides are $M$-Morse, then the third side is $M^\prime$-Morse, where $M^\prime$ depends only on $M$. For triangles where all sides are finite, this is shown in \cite{Cor16}. 

\begin{lemma}[Lemma 2.4 of \cite{Charney_2019}]\label{triangle}
Given a Morse gauge $N$ and a triangle with vertices $a$, $b$ and $c$ in $X \cup\m{X}$, there exists a Morse gauge $N'$ such that if two sides $[a, b]$, $[a, c]$ are $N$-Morse, then the third side $[b, c]$ is $N'$-Morse.
\end{lemma}

The next lemma is a technical lemma about the concatenation of quasi-geodesics. The lemma is well-known and a similar version appears for example \cite{QR22}.

\begin{lemma}\label{quasi-geodesic}
Let $p, q \in X$ and let $\gamma : [0, T]\to X$ be a continuous $(\lambda, \eps)$-quasi-geodesic. Let $t, t^\prime\in [0, T]$ such that $\gamma(t)$ and $\gamma(t^\prime)$ are points on $\gamma$ closest to $p$ and $q$ respectively. Finally, let $\alpha$ be a geodesic from $p$ to $\gamma(t)$ and let $\beta$ be a geodesic from $\gamma(t^\prime)$ to $q$. The following hold: 
\begin{enumerate}[label= \roman*)]
    \item $\alpha \cdot \gamma[t, T]$ is a $(3\lambda, \eps)$ -quasi-geodesic.
    \item If $\abs{t - t^\prime} \geq 3\lambda(d(p, \gamma(t))+d(q, \gamma(t^\prime)))$, then $\alpha \cdot \gamma[t, t^\prime]\cdot \beta$ is a $(3\lambda, \eps)$ -quasi-geodesic.
\end{enumerate}
\end{lemma}
\begin{remark}
The continuity assumption for $\gamma$ is only needed for the existence of the points $\gamma(t)$ and $\gamma(t^\prime)$.
\end{remark}

\begin{proof}
Statement i) follows directly (and is essentially equivalent to) Lemma 2.5 of \cite{QR22}.

Assume $\alpha : [0, S] \to X$ and $\beta : [0, S^\prime]\to X$. To show ii) we can use i) and thus the only thing left to prove ii) is showing that for $s\in [0, S]$ and $s^\prime\in [0, S^\prime]$ we have 
\begin{align}
    \frac{1}{3\lambda}(\abs{S-s} + \abs{t^\prime -t} + s^\prime) - \eps \leq d(\alpha(s), \beta(s^\prime))  \leq 3 \lambda (\abs{S-s} + \abs{t^\prime -t} + s^\prime) + \eps.
\end{align}
On one hand we have 
\begin{align}
    d(\alpha(s), \beta(s^\prime)) &\leq d (\alpha(s), \alpha(S)) + d ( \gamma(t), \gamma(t^\prime) ) + d (\alpha(0), \alpha(s^\prime))\\
    &\leq \lambda \abs{t - t^\prime} + \eps + \abs{S-s} + s^\prime.
\end{align}
On the other hand, we have 
\begin{align}
    d(\alpha(s), \beta(s^\prime)) &\geq   d ( \gamma(t), \gamma(t^\prime) ) - d (\alpha(s), \alpha(S))  - d (\alpha(0), \alpha(s^\prime))\\
    & \geq \frac{1}{\lambda} \abs{t-t^\prime} -\eps - (\abs{S-s} + s^\prime) \geq \frac{2}{3\lambda}\abs{t-t^\prime}- \eps\\
    &\geq \frac{1}{3\lambda}(\abs{S-s} + \abs{t^\prime -t} + s^\prime) - \eps.   
\end{align}
Together, this concludes the proof.
\end{proof}

The following two lemmas state what happens to the Morseness of a geodesic when concatenating two Morse geodesics or conversely looking at a subsegment of a Morse geodesic.

\begin{lemma} \label{morse_property:concatenation_of_geodesic}
Let $\xi: [0, \infty)\to X$ be a geodesic ray such that $\xi[0, T]$ is $M$-Morse and $\xi[T, \infty)$ is $M^\prime$-Morse for some Morse gauges $M$ and $M^\prime$. Then, $\xi$ is $M^{\prime\prime}$-Morse, where $M^{\prime\prime}$ depends only on $M$ and $M^\prime$.
\end{lemma}
\begin{proof}
Let $\gamma^\prime$ be a $(\lambda, \eps)$ -quasi-geodesic with endpoints on $\xi$. We have to show that $\gamma^\prime$ stays in a bounded distance $C^\prime$ from $\xi$, where $C$ depends only on $\lambda, \eps, M$ and $M^\prime$. By Lemma~1.11 of \cite{martin2013hyperbolic} there exists a continuous $(3\lambda, \eps)$-quasi-geodesic $\gamma : [0, S]\to X$ with the same endpoints as $\gamma^\prime$ that has Hausdorff distance at most $\lambda+\eps$ from $\gamma^\prime$. It is enough to show that $\gamma$ stays in a bounded distance $C$ from $\xi$, where $C$ depends only on $\lambda, \eps, M$ and $M^\prime$.

Let $\gamma(t)$ be a point on $\gamma$ closest to $\xi(T)$ and let $\alpha$ be a geodesic from $\xi(T)$ to $\gamma(t)$. By Lemma~\ref{quasi-geodesic}, the paths $\alpha \cdot \gamma[t, S]$ and $\alpha \cdot \gamma[0, t]$ are $(3\lambda, \eps)$-quasi-geodesics. More importantly, $\alpha \cdot \gamma[t, S]$ and $\alpha\cdot \gamma[s, 0]$ have both endpoints on either $\xi[0, T]$ or $\xi[T, \infty)$. The statement follows from the Morseness of $\xi[0, T]$ and $\xi[T, \infty)$. 
\end{proof}

\begin{remark}\label{morse:continuity}
The trick to tame $\gamma^\prime$ by approximating it by a continuous quasi-geodesic $\gamma$ can often be useful. In particular, when we want to show that a geodesic $\xi$ is $M$-Morse, where $M$ depends only on some set of parameters $\mathcal{A}$, it is enough to show that any continuous $(\lambda, \eps)$-quasi-geodesic with endpoints on $\xi$ is contained in a $C$ neighbourhood of $\xi$, where $C$ depends only on $\lambda, \eps$ and $\mathcal{A}$.
\end{remark}

\begin{lemma}\label{morse:subsegments}
Let $\xi :  I \to X$ be an $M$-Morse geodesic and $J\subset I$ any interval. Then, the geodesic $\xi\vert_J : J\to X$ is $M_S$-Morse, where $M_S$ depends only on $M$. 
\end{lemma}
\begin{proof} 
This follows directly from Lemma~2.1 of \cite{Cor16}.
\end{proof}

\subsection{Morse boundary}

As a set, the Morse boundary of $X$, denoted by $\m X$, is the set of equivalence classes of Morse geodesic rays, where two Morse geodesic rays are equivalent if they have bounded Hausdorff distance. Elements $z\in \m X$ are called Morse directions. For a Morse geodesic ray $\xi$ we denote its equivalence class by $\C{\xi}\in \m X$.

\begin{definition}[Realization]
Let $z\in \m X$ be a Morse direction. A geodesic ray $\gamma : [0, \infty) \to X $ with $\C{\gamma} = z$ and $\gamma(0) = e$  is called a realization of $z$. Similarly, for an element $x\in X$ a geodesic segment $\gamma$ starting at $e$ and ending in $x$ is called a realization of $x$.
\end{definition}

We denote by $\P{M}{X}\subset X$ the set of all points $x\in X$ that have an $M$-Morse realization and we say that points $x\in \P{M}{X}$ are $M$-Morse. Note that Remark~\ref{morse:finite} implies that $ X = \cup_{M\in \mathcal{M}} \P{M}{X}$. Similarly, we define the set of all $M$-Morse directions $\M{M}{X}$ as the set of all Morse directions $z\in\m X$ that have an $M$-Morse realization and call it the $M$-Morse boundary of $X$.

\subsubsection*{Notation.} Let $z\in \M{M}{X}$ and let $\xi$ be an $M$-Morse realization of $z$. We denote by $\G{M}{X}$ the set of all $M$-Morse geodesic rays in $X$ starting at $e$. For any positive $n$ we define
\begin{align}
    \oh{M}{n}{\xi} &= \{\eta \in \G{M}{X} \mid d(\eta(t), \xi(t)) < \delta_M \text{ for all } t\in [0, n]\}, \\
    \u{M}{n}{\xi} &= \{\C\eta \mid \eta \in \oh{M}{n}{\xi}\},\\
    \oh{M}{n}{z} & = \{ \eta \in \G{M}{X} \mid \eta \in \oh{M}{n}{\xi} \text{ for all $M$-Morse realizations $\xi$ of $z$}\},\\
    \u{M}{n}{z} &= \{\C\eta \mid \eta \in \oh{M}{n}{z}\}.
\end{align}
It is shown in \cite{Cor16} that $\{\u{M}{n}{z}\}_{n\in \N}$ is a fundamental system of neighbourhoods. It is clear that $\oh{M}{n}{\xi}\subset \oh{N}{n}{\xi}$ if $M\leq N$ as then $\G{M}{X}\subset \G{N}{X}$. The same holds for $\u{M}{n}{\xi}$. However, we have to be a bit more careful with $\oh{M}{n}{z}$ and $\u{M}{n}{z}$. This is because, if $M\leq N$ then $\G{M}{X}\subset \G{N}{X}$, but there could be more $N$-Morse realizations of $z$ than there are $M$-Morse realizations and thus more conditions to satisfy. 
\begin{remark}\label{inclusion_of_neighbourhoods}
For $n\geq 12 \delta_M$, we have $\oh{M}{n+4\delta_M}{\xi}\subset \oh{M}{n}{z}\subset \oh{M}{n}{\xi}$. This holds because for two $M$-Morse realizations $\xi, \zeta$ of $z$ we have $d(\xi(t), \zeta(t))<\delta_M$ by Corollary 2.6 of \cite{Cor16}. Thus if $\gamma\in \oh{M}{n+4\delta_M}{\xi}$, then by Corollary 2.6 we have $\gamma\in \oh{M}{n}{\zeta}$, which implies the statement.
\end{remark}

Remark \ref{inclusion_of_neighbourhoods} shows that if $M\leq N$, then $\oh{M}{n+4\delta_M}{z}\subset \oh{N}{n}{z}$ (and thus the same holds for $U$).  

\begin{remark}
Corollary 2.5 of \cite{Cor16} shows that for two $M$-Morse realizations $\xi$ and $\zeta$ of some Morse direction $z$ we have $\xi\in \oh{M}{n}{\zeta}$ for all $n$. Thus $\xi\in \oh{M}{n}{z}$ which implies that $z\in \u{M}{n}{z}$.
\end{remark}

Sometimes, we also want to make statements about finite geodesics that start at $e$ and not only about geodesic rays. Therefore, we want to extend the notation to include finite geodesics. Let $x\in \P{M}{X}$
and let $\zeta$ be an $M$-Morse realization of $x$. For $n\leq d(e, x)$ we define the sets $\oh{M}{n}{\zeta}, \oh{M}{n}{x}, \u{M}{n}{\zeta}$ and $\u{M}{n}{x}$ as above. Furthermore, we say $\GG{M}{X}$ is the set of all $M$-Morse geodesics in $X$ (both rays and finite ones) starting at $e$. For an $M$-Morse realization $\xi$ (either of a point $x\in X$ or of an $M$-Morse direction $x\in \M{M}{X}$) and some positive $n$ (that, if $\xi$ is a realization of a point $x\in X$, has to satisfy $n\leq d(e, x)$) we define 
\begin{align*}
    \Oh{M}{n}{\xi} = \{\eta \in \GG{M}{X} \mid d(\eta(t), \xi(t)) < \delta_M \text{ for all } t\in [0, n]\}.
\end{align*}
The sets $\Oh{M}{n}{x}$, $\U{M}{n}{\xi}$ and $\U{M}{n}{x}$ can be defined analogously. We call these sets filled neighbourhoods.
\begin{remark}
The observation in Remark \ref{inclusion_of_neighbourhoods} also holds for the newly defined sets.
\end{remark}

The $M$-Morse boundary inherits properties of Gromov boundary of hyperbolic spaces due to the following homeomorphism proven in \cite{CH17}.

\begin{theorem}[Thm 3.14 of \cite{CH17}] If $X$ is a proper geodesic metric space, then there is a homeomorphism between $\partial \P{M}{X}$ and $\M{M}{X}$.
\end{theorem}

Here, $\partial \P{M}{X}$ denotes the Gromov boundary of $\P{M}{X}$. This is well defined as in Theorem A III) of \cite{CH17} Cordes and Hume prove that $\P{M}{X}$ is hyperbolic. In \cite{martin2013hyperbolic} III), several properties of the Gromov boundary are shown. Important for us is that the Gromov boundary is compact and metrizable and thus, in particular, Hausdorff. With the homeomorphism above, we get that the $M$-Morse boundary of a proper geodesic metric space is compact and Hausdorff. 

Now we are ready to define a topology on the Morse boundary $\m X$. Namely, consider $\m X$ as the direct limit of the $M$-Morse boundaries $\M{M}{X}$. More precisely: 
\begin{align}
    \m X = \lim_{\xrightarrow[\mathcal{M}]{}} \M{M}{X},
\end{align}
with inclusions 
\begin{align}
    i_{M,N} : \M{M}{X} \to \M{N}{X},
\end{align}
for all $M\leq N$ defined by the canonical inclusion $\M{M}{X}\subset\M{N}{X}$. Note that those inclusions are continuous, since, for any $M$-Morse direction $z$, we have $\u{M}{n+4\delta_M}{z}\subset \u{N}{n}{z}$.

Now $\M{M}{X}$ can be viewed both on its own with the topology induced by the fundamental system of neighbourhoods $\{\u{M}{n}{z}\}_{n\in \N}$ or as a subset of $\m{X}$ with the subspace topology. We show that those two topologies coincide.
\begin{lemma}
The topology on $\M{M}{X}$ induced by the fundamental system of neighbourhoods $\{\u{M}{n}{z}\}_{n\in \N}$ coincides with the subspace topology of $\M{X}{M}\subset \m X$.
\end{lemma}
\begin{proof}
To distinguish the two topologies on $\M{M}{X}$ we denote $\M{M}{X}$ with the topology induced by the neighbourhood basis by $Y$ and $\M{M}{X}\subset \m{X}$ with the subspace topology by $Z$. By definition of the direct limit topology, the inclusion $i : Y \to \m{X}$ is continuous and thus if $C\subset Z$ is closed, so is $C\subset Y$.

Conversely, let $C\subset Y$ be closed. Let $N\geq M$ be a Morse gauge. Then, as argued above, the inclusion $i_{M, N} : Y\to \M{N}{X}$ is continuous. As $Y$ and $\M{N}{X}$ are compact and Hausdorff, $i_{M, N}(C) = C\subset \M{N}{X}$ is closed. Now, let $N$ be any Morse gauge and $N^\prime = \max\{N, M\}$. Then $C\cap \M{N}{X}=i_{N, N^\prime}^{-1}(C)$ is closed. This concludes the proof, as it shows that $C\subset \m X$ is closed and thus also $C\subset Z$ is closed. 
\end{proof}

The Morse boundary is the direct limit of compact sets. However, a priori, as it is the direct limit of uncountably many compact sets it might not even be $\sigma$-compact. However, in \cite{CD16} they characterized compact subsets of $\m X$: 

\begin{lemma}[Lemma 4.1 of \cite{CD16}]\label{morse:compact_subsets} If $K\subset \m X$ is non-empty and compact, then there exists $M>0$ such that $K\subset \M{M}{X}$.
\end{lemma} 

In fact, it is proven in \cite{CD16} that if the Morse boundary of a group is non-empty and compact, then the group is hyperbolic. 

\begin{corollary}[Corollary of Theorem 1.1 of \cite{CD16}]\label{cor:compact_implies_hyperbolic}
Let $G$ be a finitely generated infinite group. The Morse boundary of $G$ is non-empty and compact if and only if $G$ is hyperbolic.
\end{corollary}

\begin{proof}
Theorem 1.1. of \cite{CD16} states that a subgroup $H$ of a finitely generated group $G$ is boundary convex cocompact if and only if it is stable in $G$. One can readily see that the group $G$ is stable in itself if and only if it is hyperbolic. Furthermore, $G$ viewed as a subgroup of itself is boundary convex cocompact if and only if the Morse boundary of $G$ is compact and non-empty.

Hence, the corollary follows by applying Theorem 1.1 of \cite{CD16} with $ H = G$.  
\end{proof}

One implication of Corollary \ref{cor:compact_implies_hyperbolic} is the following. 

\begin{lemma}\label{lemma:not_one}
Let $G$ be a finitely generated group. Either $\m G = \emptyset$ or $\abs{\m G}\geq 2$.
\end{lemma}

\begin{proof}
Assume by contradiction that $\abs{\m G} = 1$. Since $\m G$ is finite, it is compact. Hence $G$ is hyperbolic, which implies that the Morse boundary of $G$ coincides with the Gromov boundary of $G$ by \cite{CH17}. Proposition 7.15 of \cite{ghys1990espaces} shows that $\abs{\partial G}\neq 1$ for any hyperbolic group $G$. This is a contradiction to the assumption $\abs{\m G} = 1$. 
\end{proof}

Next, we prove two technical lemmas about the Morse boundary. The first one is about the Morseness of a realization of an $M$-Morse direction - a priori this might have nothing to do with the Morse gauge $M$. The second one is about properties of the neighbourhoods defined above. 

\begin{lemma}\label{morse:bound_on_realization}
Any realization $\xi$ of an $M$-Morse direction $z\in \M{M}{X}$ is $M_{\mathcal{C}}$-Morse, where $M_{\mathcal{C}}$ depends only on $M$.  
\end{lemma}

\begin{proof}
Let $\zeta$ be an $M$-Morse realization of $z$. Cor 2.6 of \cite{Cor16} implies that $d(\zeta(t), \xi(t))< \delta_M$ for all $t\in [0, \infty)$. Lemma~2.5 (i) of \cite{Charney_2014} concludes the proof.
\end{proof}

\begin{lemma}\label{morse_property:open_set_in_open_set}
Let $z\in \M{M}{X}$ be an $M$-Morse direction. For any integer $l\in \N$ there exists a positive integer $k$ such that for all $y\in \U{M}{k}{z}$ we have $\U{M}{k}{y}\subset\U{M}{l}{z}$.
\end{lemma}

\begin{remark} If we would allow a dependency of $k$ on $y$, then the lemma is a direct consequence of the $\U{M}{l}{z}$ being a fundamental system of neighbourhoods. The crucial part of this lemma is that there exists a uniform $k$, that is, there exists a choice of $k$ that works for all $y\in \U{M}{k}{z}$.
\end{remark}

\begin{proof}
Set $k = \max \{l + 2K, 6K\}$, where $K= 2\delta_M$. Let $\xi$ be an $M$-Morse realization of $z$, let $y\in \U{M}{k}{z}$, and let $\eta\in \Oh{M}{k}{z}$ be a realization of $y$. Furthermore, let $\lambda\in \Oh{M}{k}{y}$. We have that $d(\xi(t), \eta(t)) < \delta_M$ and $d(\eta(t), \lambda(t))<\delta_M$ for all $t\in[0, k]$ and hence $d(\xi(t), \lambda(t))<2\delta_M$ for all $t\in [0,k]$. Using Cor~2.5 of \cite{Cor16} we get that $d(\xi(t), \lambda(t))<\delta_M$ for all $t\in [0, l]$. As this is true for any $M$-Morse realization $\xi$ of $z$ we have $\lambda\in \Oh{M}{l}{z}$ and thus $\C\lambda\in \U{M}{l}{z}\subset U$.
\end{proof}

\subsection{Morse geodesic lines} \label{section:morse_geodesic_lines}

Let $A$ be a finitely generated group with neutral element $e$. For any finite set of generators $S\subset A$, the Cayley-graph $\Gamma(A, S)$ is a proper geodesic metric space. Furthermore, for two finite generating sets $S_1$ and $S_2$, the two Cayley graphs $\Gamma(A, S_1)$ and $\Gamma(A, S_2)$ are quasi-isometric. In \cite{Cor16} it is proven that the Morse boundary is a quasi-isometry invariant. Therefore, we can define the Morse boundary of $A$ as the Morse boundary of $\Gamma(A, S)$ for any finite generating set $S$ and the homeomorphism type of $\m A$ does not depend on the choice of $S$. Sometimes, by abuse of notation, we will implicitly identify $A$ with its Cayley graph, assuming that we fixed a set of generators $S$. The base point of $A$ (resp its Cayley graph) will be assumed to be its neutral element $e$.

\subsubsection*{Group action on the Morse boundary} The left action of $A$ by isometries on its Cayley graph induces a group action of $A$ on the Morse boundary by $a \cdot \C\gamma = \C{a\cdot \gamma}$. Also observe that if $\gamma$ is $M$-Morse, so is $a\cdot \gamma$ (but not necessarily $\C{a\cdot \gamma}$).

In the following, we assume $\abs{\m A}\geq 2$. Then, we fix a Morse geodesic line going through $e$, denoting it by $\lambda(A)$ and denote by $M_0$ a Morse gauge such that $\lambda(A)$ is $M_0$-Morse.

\begin{remark} Note that by a geodesic line, we mean a bi-infinite geodesic. Such a geodesic line $\lambda(A)$ exists: The Morse boundary is a visibility space (see \cite{Cor16}) and thus having $\abs{\m A}\geq 2$ implies that there exists a Morse geodesic line $\tilde{\lambda}$ going through some point $x \in A$. Then $x^{-1}\cdot \tilde{\lambda}$ is a Morse geodesic line going through the base point $e$.
\end{remark}

For any element $x\in A$ we construct a Morse geodesic $\lambda_x$ corresponding to $x$ as follows: Let $\tilde{\lambda}$ denote the $M_0$-Morse geodesic line $x \cdot \lambda (A)$ and let $x_0\in A$ be a point on $\tilde{\lambda}$ closest to $e$. 
We may (and will) assume that $\tilde{\lambda}(0) = x_0$ (if not reparametrize $\tilde\lambda$) and that the integer $t_x$ such that  $\tilde{\lambda} (t_x) = x$ is non-negative (if not replace $\tilde{\lambda}$ with its mirrored geodesic line, i.e. the geodesic line $\tilde{\lambda}^{\prime}$ that satisfies $\tilde{\lambda}^{\prime}(t)= \tilde{\lambda}(-t)$ for all $t\in \R$ ).
Define $\lambda_x$ to be any realization of $\C{\tilde{\lambda}[0, \infty)}$ and call it the geodesic ray corresponding to $x$. The construction of $\lambda_x$ is depicted in Figure \ref{picture:def_of_lambda_x}.

\begin{figure}\centering
\includegraphics[width= \linewidth]{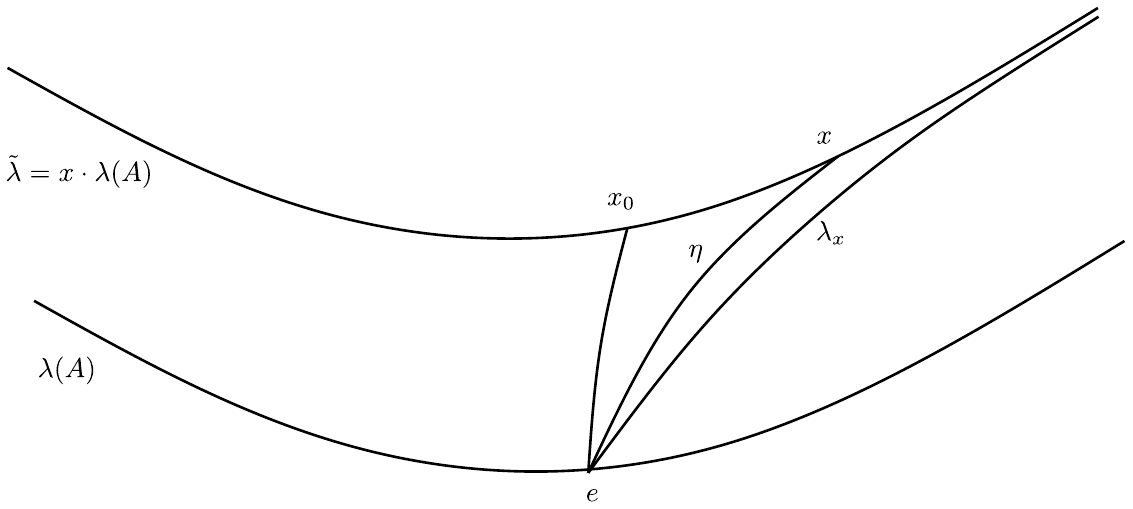}
\caption{Construction of the to $x$ corresponding geodesic ray $\lambda_x$.}
\label{picture:def_of_lambda_x}
\end{figure}

Intuitively speaking, the Morseness of $\lambda(A)$ gives us a base Morseness $M_0$ and $\lambda_x$ captures the Morse properties of $x$ up to the base Morseness. The lemma below states this more precisely:

\begin{lemma} \label{map:morseline:porperties}
Let $x\in A$ and let $\eta$ be an $M$-Morse realization of $x$. The following hold:
\begin{enumerate}[label = \roman*)]
\item The geodesic ray $\lambda_x$ corresponding to $x$ is $M_{\lambda}$ -Morse, where $M_{\lambda}$ depends only on $M$ and $M_0$.\label{map:morseline:properties:1}
\item If $\lambda_x$ is $N$-Morse, then $x\in\P{N^\prime}{X}$, where $N^\prime$ depends only on $N$ and $M_0$. \label{map:morseline:properties:4}
\item For any positive integer $T$, there exists a constant $C_{M}(T)$, depending only on $M$, $M_0$ and $T$, such that if $d(e, x)\geq C_M(T)$, then $d(\lambda_x(t), \eta(t)) < \delta_M$ for all $t\in [0, T]$. \label{map:morseline:properties:2}
\item There exists an integer $T_1$ such that $\lambda_x([T_1, \infty))$  is $M_1$-Morse, where $M_1$ depends only on $M_0$. \label{map:morseline:properties:3}
\end{enumerate}
\end{lemma}

\begin{remark}\label{lemma1.16_direct}
Lemma~\ref{map:morseline:porperties} implies that if $d(e, x)\geq C_M(T)$, then $x$ is in $\U{M_\lambda}{T}{\lambda_x}$ and $\lambda_x$ is in $\u{M_\lambda}{T}{x}$. Together with  Remark~\ref{inclusion_of_neighbourhoods} we also have that if $d(e, x) \geq C_M(T + 4\delta_{M_\lambda})$, then $x$ is in $\U{M_\lambda}{T}{\C{\lambda_x}}$.
\end{remark}
\begin{proof}
The statements i) and ii) follow from Lemma \ref{triangle}, because $\Delta = \left(\eta,  \tilde{\lambda} [0, \infty), \lambda_x\right)$ forms a generalized triangle. 

For any $t\in [0, \infty)$ let $\tilde{\lambda}(s_t)$ be a closest point on $\tilde{\lambda}$ to $\lambda_x(t)$. Since $\tilde{\lambda}[0, \infty)$ and $\lambda_x$ represent the same Morse direction, they have finite Hausdorff distance $D$. In particular, $d(\tilde{\lambda}(s_t), \lambda_x(t)) \leq D$ for all $t\in [0, \infty)$. Hence, Lemma \ref{quasi-geodesic} ii) implies that for large enough $t$, the concatenation $[e, x_0][\tilde{\lambda}[0, s_t][\tilde{\lambda}(s_t),\lambda_x(t)]$ is a $(3, 0)$-quasi-geodesic. Since, for large enough $t$, $x$ lies on this concatenation $d(x, \lambda_x) \leq M_\lambda(3, 0)< \delta_{M_\lambda}$. Thus, as $\eta$ and $\lambda_x$ are geodesics, we have $d(\eta(t), \lambda_x(t))<3\delta_{M_\lambda}$ for all $t\in [0, d(e, x)]$. Set $C_M(T)= \max\{18\delta_{M_\lambda}, T+ 6 \delta_{M_\lambda}\}$ and assume $d(e, x)\geq C_M(T)$. According to Corollary 2.5 of \cite{Cor16}, we have $d(\eta(t), \lambda_x(t))<\delta_M$ for all $t\in [0, T]$. This concludes the proof of iii).

For any $s\in [0, \infty)$ let $\lambda_x(t_s)$ be the point on $\lambda_x$ closest to $\tilde{\lambda}(s)$. Define 
\begin{align*}
    \gamma_s = [\tilde{\lambda}(0), \lambda_x(t_0)]\cdot \lambda_x[t_0, t_s]\cdot[\lambda_x(t_s), \tilde{\lambda}(s)].
\end{align*}
By Lemma~\ref{quasi-geodesic} the path $\gamma_t$ is a $(3, 0)$-quasi-geodesic for large enough $s$. Increasing $s$ we have, by Lemma~2.1 of \cite{Cor16}, that
\begin{align*}
    d_H([\tilde{\lambda}(0), \lambda_x(t_0)]\lambda_x[t_0, \infty), \tilde{\lambda}[0, \infty) )\leq 2 M_0(3, 0).
\end{align*}
Choose $T_1, T_2$ large enough such that $d(\lambda_x(T_1), \tilde{\lambda}(T_2))\leq 2 M_0(3, 0)$, then by Cor 2.6 of \cite{Cor16} we have $d(\lambda_x(T_1+s), \tilde{\lambda}(T_2+s))\leq 2M_0(3, 0) +\delta_{M_0}$. So the Hausdorff distance between $\lambda_x[T_1, \infty)$ and $\tilde{\lambda}[T_2, \infty)$ is bounded by some constant depending only on $M_0$. Lemma 2.5 of \cite{Charney_2014} together with Lemma~\ref{morse:subsegments} conclude the proof of iv). 
\end{proof}

\subsection{Homeomorphisms of Morse boundaries}\label{section:homeo_of_morse_boundaries} 

To prove Theorem \ref{freeproduct_theorem} we want to construct a homeomophism $\bar{p} : \m{(A_1*B_1)}\to \m {(A_2 *B_2)}$ given homeomorphisms $p_A : \m{A_1}\to \m {A_2}$ and $p_B : \m{B_1}\to \m{B_2}$. To construct the homeomorphism $\bar{p}: \m{(A_1*B_1)}\to \m {(A_2 *B_2)}$ we first construct bijections $\hat{p}_A : A_1\to A_2$ and $\hat{p}_B : B_1 \to B_2$ that induce the homeomorphisms $p_A$ and $p_B$. 

Hence, it is important to understand how homeomorphisms of Morse boundaries interact with the direct limit topology and how we can prove that a map between two Morse boundaries is continuous. Also, we will formalize what it means for a bijection of the groups to induce a homeomorphism of the boundaries. We start with describing some properties of homeomorphisms of the Morse boundray. 

Let $G$ and $H$ be finitely generated groups with basepoint $e$ and let $p: \m G \to \m H$ be any homeomorphism. Recall that $\mathcal{ M}$ denotes the set of all Morse gauges and note that by abuse of notation, we write $p(\xi)$ for $p(\C{\xi})$ for any geodesic ray $\xi$. The homeomorphism $p$ interacts nicely with the direct product structure of the Morse boundary since $p$ maps $M$-Morse boundaries into $M'$-Morse boundaries. More precisely, there exists a map $h: \mathcal{M}\to \mathcal{ M}$, such that for ever Morse gauge $M$ 
\begin{align}\label{eq:homoe_of_morse_strata}
    p(\M{M}{G})\subset \M{h(M)}{H} \qquad \text{and} \qquad  p^{-1}(\M{M}{H})\subset \M{h(M)}{G}.
\end{align}

The existence of such a map $h$ is due to the compactness of the $M$-Morse boundaries. More precisely, since $\M{M}{G}$ is compact, so is $p(\M{M}{G})$. Lemma~\ref{morse:compact_subsets} then implies that $p(\M{M}{G})\subset \M{N}{H}$ for some Morse gauge $N$. We can argue analogously for the inverse $p^{-1}$ and thus such a map $h$ exists.

Conversely, let $q: \m G\to \m H$ be a map which we want to prove that is continuous. As $\m G$ and $\m H$ are constructed via direct limits, it is enough to show that $q$ is continuous on every $M$-Morse boundary. More precisely, if there exists a map $h: \mathcal{M}\to \mathcal{M}$ such that $q\vert_{\M{M}{G}} : \M{M}{G}\to \M{h(M)}{H}$ is continuous for all $M$, then so is $q$ itself. The topology of $\M{M}{G}$ and $\M{h(M)}{H}$ are given by a fundamental system of neighbourhoods. Hence if we can prove that there exists a map $h: \mathcal M \to \mathcal M$ such that for every integer $k$ there exists an integer $l$ such that 
\begin{align}
    q(\u{M}{l}{x})\subset \u{h(M)}{k}{q(x)},
\end{align} 
then we have shown $q$ is continuous.

Next, we will define what it means for a bijection $\hat{p} : G\to H$ to induce a homeomorphism $p : \m G\to \m H$. To do so, we first need to define what it means for a sequence  $(x_n)_n$ of points in $G$ to converge to a Morse direction $z\in \m G$. Recall that the filled neighbourhood $\U{M}{k}{z}$ contains both Morse directions and elements of $G$.

\begin{definition}[Convergence] A sequence $(x_n)_n$ of $M$-Morse points in $G$ converges to a Morse direction $z\in \M{M}{G}$ if for every integer $k$ all but finitely many elements of the sequence $(x_n)_n$ lie in $\U{M}{k}{z}$.
\end{definition}

Note that there is a subtle difference between this definition of convergence for a sequence of points $(x_n)_n$ and the (uniform) convergence (on compact sets) of geodesics $([e, x_n])_n$, the latter one being strictly stronger. Namely, if the sequence of geodesics $([e, x_n])_n$ converges to some geodesic ray $\xi$, then the sequence of points $(x_n)_n$ converges to $[\xi]$. Conversely, if $(x_n)_n$ converges to some Morse direction $z$, then $([e, x_n]_n)$ might not converge. But, if $([e, x_n])_n$ (or any subsequence of it) converges to some geodesic ray $\xi$, then $[\xi] = z$. The reason for the possible lack of convergence of $([e, x_n])_n$ is that for geodesic rays to converge, they need to come uniformly close on compact sets for large $n$, where close means closer than $\eps$ for any choice of $\eps>0$. However, for the sequence $(x_n)_n$ to converge, it is enough that the geodesics in $([e, x_n])_n$ come $\delta_M$ close for large $n$. 

\begin{definition}
A bijection $\hat{p} : G\to H$ induces a homeomorphism $p : \m G\to \m H$ if the following holds. For any Morse gauge $M$ and any sequence $(x_n)_n$ (resp $(x_n')_n$) of $M$-Morse points in $G$ (resp in $H$) converging to some Morse direction $z$ (resp $z'$) we have that the sequence $(\hat{p}(x_n))_n$ (resp $(\hat{p}^{-1}(x'_n))_n$) converges to $p(z)$ (resp $p^{-1}(z')$).
\end{definition}

Lastly, we prove a technical lemma which will be used in Section \ref{chapter:map}. Let $p: \m G\to \m H$ be a homeomorphism and $h : \mathcal M\to \mathcal M$ a map such that \eqref{eq:homoe_of_morse_strata} holds. The lemma states that given Morse direction $z$, if you take a point $x$ far along the realization of $p(z)$, then $p^{-1}(\lambda_x)$ is close to $z$ for a long time. 

\begin{lemma} \label{homeo_properties:back_match_lambda_x}
Let $p: \m G\to \m H$ be a homeomorphism and $h : \mathcal M\to \mathcal M$ a map such that \eqref{eq:homoe_of_morse_strata} holds. Let $z\in \M{M}{G}$ be an $M$-Morse direction and $\eta$ be an $h(M)$-Morse realization of $p(z)$. For any positive integer $n$, there exists an integer $i_0$ such that 
\begin{align}\label{cond:far_along}
    p^{-1} (\lambda_{\eta(i)}) \in \U{h(h(M)_\lambda)}{n}{z},
\end{align}
for all $i\geq i_0$.
\end{lemma}
\begin{proof}
There exists an integer $k$ such that 
\begin{align}
    p^{-1}(\u{h(M)_\lambda}{k}{p(z)})\subset \u{h(h(M)_\lambda)}{n}{z}.\label{proof_of_3.1_eq1}
\end{align}
Choose $i_0 = C_{h(M)}(k + 4\delta_{h(M)_\lambda})$. For any $i\geq i_0$ the geodesic $\eta([0, i])$ is an $h(M)$-Morse realization of $\eta(i)$. Remark~\ref{lemma1.16_direct} shows that $\lambda_{\eta(i)}\in \oh{h(M)_\lambda}{k}{p(z)}$. The proof is concluded by \eqref{proof_of_3.1_eq1}.
\end{proof}

\section{The Morse boundary of the free product}\label{chapter:comb}

Let $A$ and $B$ be finitely generated groups. In this section, we give an alternative description of the Morse boundary of $A\ast B$, call it the combinatorial Morse boundary and show that it is homeomorphic to $\m{(A\ast B)}$. To do so, we fix Cayley graphs $\Gamma_A = \Gamma(A, S_A)$, $\Gamma_B = \Gamma(B, S_B)$ and $\Gamma = \Gamma(A\ast B, S_A\cup S_B)$ for some finite generating sets $S_A\subset A$ and $S_B\subset B$. With this notation, $\Gamma_A$ and $\Gamma_B$ can be isometrically embedded into $\Gamma$ using the inclusions $A\into A\ast B$ and $B\into A\ast B$. There is also a different way to think about $\Gamma$. Namely, let $Y_A = \Gamma_A\times (A*B)$ be a family of copies of $\Gamma_A$, where there is a copy for each element of $A*B$. Similarly, let $Y_B = \Gamma_B\times (A*B)$. Then $\Gamma$ is a quotient of $Y_A\sqcup Y_B$ under the following equivalence relation: for every pair of vertices $v\in \Gamma_A$ and $w\in \Gamma_B$ and pair elements $g, h\in A*B$ we have that $(v, g)\sim(w, h)$ if and only if $gv = hw$. With this definition we see that the inclusion $\Gamma_A\into\Gamma$ is given by $x\mapsto (x, 1)$.

\subsubsection*{Notation:} For any set $C\subset A\ast B$ we denote by $\Gamma[C]$ the graph that contains all vertices of $C$ and all edges that have at least one endpoint in $C$. 

A thorough discussion of amalgams and free products can be found in \cite{serre2012trees}. Important for us is that any element $w\in A\ast B$ has a unique reduced representation
\begin{align}
    w = a_1 b_2\cdot \ldots \cdot a_{k} b_{k},
\end{align}
where $k\geq 1, a_i\in A$ and $b_i\in B$ for all $1\leq i \leq k$ and furthermore $a_{i}\neq e$ for all $i>1$, and similarly $b_{i} \neq e$ for all $i<k$.  If $b_{k} = e$ we say the length of $w$ is $2k-1$, otherwise, we say the length of $w$ is $2k$. Let $w$ be an element of $A\ast B$. Assume that $w = a_1 b_1\cdot \ldots \cdot a_k b_k$ has length $l$ in its reduced representation. We denote by $\Gamma_w$ the graph $\Gamma\setminus\{w\}$ and we analyze its connected components. Let $C_w\subset A\ast B$ be the set of all elements whose reduced representations have length $\geq l+1$ and start with $a_1 b_1\cdot \ldots \cdot a_k b_k$. Let $\overline{C}_w = A\ast B\setminus (C_w \cup \{w\})$. Observe that there exists no edge between $C_w$ and $\overline{C}_w$. Therefore, any geodesic between $C_w$ and $\overline{C}_w$ goes through $w$. With the same argumentation, we get:

\begin{lemma}\label{combinatorial:geodesics}
Any path from $e$ to $w$ goes through $a_1, a_1b_1, \ldots ,a_1b_1\cdot\ldots \cdot b_{k-1}$ and $a_1b_1\cdot\ldots \cdot a_k$. 
\end{lemma}

Observe that for $a\neq a^\prime\in A$ the sets $C_a$ and $C_{a^\prime}$ are disjoint. Therefore, we can partition $A\ast B\setminus A$ into sets $\{C_a \}_{ a\in A}$. This gives rise to a projection $\tilde{\pi} : A\ast B \to A$ defined by $\tilde{\pi}(a) = a$ for all $a\in A$ and $\tilde{\pi}(w) = u$ for the unique $u\in A$ such that $w\in C_{u}$. The projection $\tilde\pi$ induces a projection $\pi : \Gamma \to \Gamma_A$ defined by $\pi(x) = x $ if $x\in \Gamma_A$ and $\pi (x ) = \tilde{\pi}(w)$ if $x$ is connected to $w \in A\ast B$ in $\Gamma \setminus \Gamma_A$. This is well defined: if $w$ and $v$ are connected in $\Gamma \setminus \Gamma_A$, then $\tilde{\pi}(w) = \tilde{\pi}(v)$.

\begin{lemma}
The projection $\pi$ as defined above is continuous.
\end{lemma}

\begin{proof}
We have to prove that for all $x\in \Gamma$ and $\eps >0$ there exists $\delta>0$ such that if $d(x, y)<\delta$, then $d(\pi(x), \pi(y))<\eps$. If $x\in \Gamma_A$, choose $\delta = \eps$. This works as $\pi(x) = x$ and any path from $x$ to $y$ goes through $\pi(y)$. If $x\notin\Gamma_A$ choose $\delta = d(x, \pi(x))$. Any $y$ with $d(y, x)<\delta$ then satisfies that $y$ and $x$ are in the same connected component of $\Gamma \setminus \Gamma_A$ and thus $\pi(x) = \pi(y)$. 
\end{proof} 

\begin{lemma}\label{free_product:moving_geo_to_gamma_a}
Let $\gamma$ be a continuous $(\lambda, \eps)$-quasi-geodesic in $\Gamma$ with endpoints in $\Gamma_A$. Then $\pi\circ \gamma : [a, b]\to \Gamma_A$ is a continuous $(\lambda, \eps)$-quasi-geodesic with $d_H(\gamma, \pi\circ\gamma)\leq \lambda^2\eps+\eps$.   
\end{lemma}

\begin{proof}
The map $\pi \circ \gamma$ is continuous because $\pi$ and $\gamma$ are continuous. To show that $\pi\circ\gamma$ is a $(\lambda, \eps)$-quasi-geodesic, we have to prove that
\begin{align}
    \frac{1}{\lambda}\abs{t-s}-\eps \leq d(\pi\circ\gamma(t), \pi\circ\gamma(s) )  \leq \lambda \abs{t-s}+\eps,\label{lem2.3eq2}
\end{align}
for all $t, s\in [a, b]$. The right hand side follows from 
\begin{align}
d(\pi\circ\gamma(t), \pi\circ\gamma(s) )\leq d(\gamma(t), \gamma(s)), \label{lem2.3eq1}
\end{align}
because $\gamma$ is a $(\lambda, \eps)$-quasi-geodesic. Note that $\eqref{lem2.3eq1}$ holds because if $\gamma(t)$ and $\gamma(s)$ are in the same connected component of $\Gamma\setminus\Gamma_A$, then $\pi\circ\gamma(t) = \pi\circ\gamma(s)$. If $\gamma(t)$ and $\gamma(s)$ are not in the same connected component of $\Gamma\setminus  \Gamma_A$, then any path (and thus any geodesic) from $\gamma(t)$ to $\gamma(s)$ goes through $\pi\circ\gamma(t)$ and $\pi\circ\gamma(s)$.

To prove the left hand side of \eqref{lem2.3eq2}, we need the following observation: if $\pi\circ\gamma(s) = a$ then there exists $s_1\leq s\leq s_2$ such that $\gamma(s_1) = x$ and $\gamma(s_2) = x$. This holds, as any path from $\gamma(a)$ to $\gamma(s)$ goes through $x$. Analogously any path from $\gamma(s)$ to $\gamma(b)$ goes through $x$.
Let $s, t\in [a, b]$ and $s^\prime \leq s \leq t \leq t^\prime$ such that $\gamma(s^\prime) = \pi(\gamma(s))$ and $\gamma(t^\prime) = \pi(\gamma(t))$. As $\gamma$ is a $(\lambda, \eps)$-quasi-geodesic, we have $d(\gamma(s^\prime), \gamma(t^\prime))\geq \frac{1}{\lambda}\abs{s^\prime - t^\prime} - \eps$. But $d(\gamma(s^\prime), \gamma(t^\prime) )= d(\pi\circ\gamma(t), \pi\circ\gamma(s) )$ and $\abs{s^\prime - t^\prime}\geq \abs{s-t}$, which concludes the proof of the left hand side of \eqref{lem2.3eq2}.

To show that $d_H(\gamma, \pi\circ\gamma)\leq \lambda^2\eps+\eps$ it is enough to show that $d(\gamma(t), \pi\circ\gamma(t))\leq \lambda^2\eps +\eps$ for all $t\in [a, b]$. Let $t\in [a, b]$ and $t_1\leq t\leq t_2$ such that $\gamma(t_1) = \pi\circ\gamma(t) = \gamma(t_2)$. We have shown above that $d(\gamma(t_1) , \gamma(t))\leq \lambda\abs{t_1-t} + \eps$. We also have $0=d(\gamma(t_1), \gamma(t_2) )  \geq  \frac{1}{\lambda}\abs{t_1-t_2}-\eps$. Therefore, $\abs{t_1-t}\leq\abs{t_1-t_2}\leq \lambda\eps$, which concludes the proof.
\end{proof}

\begin{lemma}\label{combinatorial:move_geodesics}
Let $\eta\subset \Gamma_A$ be an $M$-Morse geodesic. Then, $\eta$ considered as a geodesic in $\Gamma$ is $M^\prime$-Morse, where $M^\prime$ depends only on $M$.
\end{lemma}
\begin{proof}
Let $\gamma\subset \Gamma$ be a continuous $(\lambda, \eps)$-quasi-geodesic with endpoints on $\eta$.  Lemma~\ref{free_product:moving_geo_to_gamma_a} shows that $\pi\circ\gamma\subset\Gamma_A$ is a $(\lambda, \eps)$-quasi-geodesic with endpoints on $\eta$. Since $\eta$ is $M$-Morse in $\Gamma_A$, $\pi\circ\gamma$ is in an $M(\lambda, \eps)$ neighbourhood of $\eta$. Thus $\gamma$ is in an $M(\lambda, \eps)+ \lambda^2\eps + \eps$ neighbourhood of $\eta$. By Remark~\ref{morse:continuity}, this concludes the proof. 
\end{proof}

\subsection{Construction of the combinatorial Morse boundary}\label{section2.1}

Consider an $M$-Morse geodesic ray $\xi$ in $\Gamma$ starting at $e$. The geodesic $\xi$ travels along an infinite sequence of edges $(s_i)_{i\in \N}$, each edge being from either $S_A$ or $S_B$. One of the following happens: either $(s_i)_{i\in \N}$ contains infinitely many edges from  both $S_A$ and $S_B$ (we call this the infinite case) or the tail of $(s_i)_{i\in \N}$ is contained completely in either $S_A$ or $S_B$ (we call this the finite case). Define $c_1 = s_1\cdot \ldots \cdot s_{k_1}$ such that $k_1$ is the maximal integer satisfying $s_i\in A$ for all $1\leq i \leq k_1$. For $i\geq 2$ iteratively define 
\begin{align}
    c_{i} &= s_{k_{i-1}+1}\cdot \ldots \cdot s_{k_{i}}
\end{align}
where $k_i$ is the maximal integer such that $s_j\in G_{i\%2}$ for all $k_{i-1}< j \leq k_i$, where $G_{i\%2}$ denotes $A$ if $i$ is odd and $B$ otherwise. In the infinite case $c_{2i+1}\in A\setminus \{e\}$, $c_{2i}\in B\setminus\{e\}$ for all $i\geq 1$ and $\xi$ walks from $c_1$ to $c_1c_2$ to $c_1 c_2 c_3$ and so on. In the finite case we have some $i$ such that $c_i$ is defined as the product of infinitely many elements of $G_{i\%2}$. Then, $c_i$ can not be viewed as an element of $G_{i\%2}$ but as a geodesic in $\Gamma_{G_{i\%2}}$ starting at $e$ and going along $s_{k_{i-1}+1}, s_{k_{i-1}+2}, \ldots$ and so on. 

This shows that in the infinite case, we can view $\xi$ as an infinite sequence $(c_i)_{i\in\N}$ and in the finite case we can view $\xi$ as a sequence $c_1, \ldots , c_m$ and some geodesic $\xi^\prime\in \Gamma_{G_{(m+1)\%2}}$. It also motivates the following definition: Let $M$ be a Morse gauge, define
\begin{align*}
    \T{M}  &= \left\{ (u_i)_{i\in \N} \mid u_i\in \begin{cases}
    \P{M}{A} & \text{if $i = 1$,}\\
    \P{M}{A}\setminus \{\e\} & \text{if $i>1$ is odd,}\\
    \P{M}{B}\setminus\{\e\}& \text{if $i$ is even.} 
    \end{cases}\qquad \right\}\\
    \stab{M}  &= \left\{ (u_1, \ldots, u_n ; z ) \mid n\in \N_0,\quad  u_i\in \begin{cases}
    \P{M}{A} & \text{if $i = 1$,}\\
    \P{M}{A}\setminus \{\e\} & \text{if $i>1$ is odd,}\\
    \P{M}{B}\setminus\{\e\} & \text{if $i$ is even.} 
    \end{cases}\quad 
    z \in \M{M}{G_{n\%2}}\right\}.
\end{align*}
Going back to our description from above we have in the infinite case $(c_i)_{i\in \N}\in \T{M_S}$ and in the finite case $(c_1, \ldots, c_{i-1}; \C{c_i})\in \stab{M_S}$.

As a set, the combinatorial $M$-Morse boundary $\co{M}$ is defined as
\[
\co{M} = \T{M} \cup \stab{M}. 
\]

We will call elements of $\co{M}$ combinatorial $M$-Morse geodesics and say that an element $a\in\co{M}$ is:
\begin{itemize}
    \item Of \emph{infinite type} if $a\in \T{M}$.
    \item Of \emph{finite type} if $a\in\stab{M}$.
\end{itemize}

Sometimes, we will talk about elements $\co{M}$ without writing down its explicit representation. To nonetheless be able to talk about explicit parts of their representation, we introduce the following notation.
\subsubsection*{Notation} Let $ a = (u_1, \ldots, u_n ; z)$ be an element of $\stab{M}$. We say the integer $n$ is the \emph{length of $a$} and we will denote it by $l(a)$. For any integer $1\leq i \leq l(a)$, $u_i(a)$ will denote $u_i$ and $\gamma(a)$ will denote $z$. For an element $b= (u_i)_{i\in\N}$ of $\T{M}$ we will say that its length $l(b)$ is infinite and for any positive integer $i$, $u_i(b)$ will denote $u_i$. Furthermore, define $v_0(a) = e$ and $v_i(a) = u_i v_{i-1}(a)$ for all $1\leq i \leq l(a)$.

Let $M$ be a Morse gauge. We will define the topology on $\co{M}$ by defining a fundamental system of neighbourhoods $\{V^{M}_k(a)\}_{k\in \N}$ for all combinatorial $M$-Morse geodesics $a\in\co{M}$. Let $k\in \N$ be an integer and $a\in\co{M}$:
\begin{itemize}
    \item If $a$ is of infinite type we will define 
    \begin{align}
        V_k^M (a) = \{b \in \co{M} \mid \text{ $l(b) \geq k$ and $u_i(b) = u_i(a)$ for all $1\leq i  \leq k$} \}.
    \end{align}
    \item If $a$ is of finite type, $V_k^M(a)$ consists of all combinatorial $M$-Morse geodesics $b\in \co{M}$ that satisfy:
    \begin{enumerate}[label = \roman*)]
        \item $l(b)\geq l(a)$, \label{haha}
        \item $u_i(b) = u_i(a)$, for all $1\leq i \leq l(a)$,
        \item $u_{l(a)+1} \in \U{M}{k}{\gamma(a)}$, if $l(b) > l(a)$,
        \item $\gamma(b) \in U_k^M(\gamma(a))$, if $l(b) = l(a)$.
    \end{enumerate}
\end{itemize}
\begin{remark}
Condition ii) is equivalent to the condition $v_{l(a)}(a) = v_{l(a)}(b)$.
\end{remark}
\begin{lemma}
This is indeed a fundamental system of neighbourhoods.
\end{lemma}
\begin{proof}
Let $a$ be a combinatorial $M$-Morse geodesic. We need to show the following three properties hold: 
\begin{enumerate}
    \item For all $k\in\N$, we have $a\in V_k^M(a)$.
    \item For all $i, j\in \N$ there exists a positive integer $k$ such that $V_k^M(a)\subset V_i^M(a)\cap V_j^M(a)$.
    \item For all $V_i^M(a)\in \{V_n^M(a)\}_{n\in \N}$ there exists $V_j^M(a)\in \{V_n^M(a)\}_{n\in \N}$ such that for every $b\in V_j^M(a)$ there exists $k\in \N$ such that $V_k^M(b)\subset V_i^M(a)$.
\end{enumerate}

Property 1 follows directly from the definition. Property 2 holds when setting $k= \max\{i, j\}$. To prove Property 3, we will make a case distinction: Assume $a$ is of infinite type. In this case, setting $j=i+1$ and $k=j$ works.

Assume $a$ is of finite type. Use Lemma \ref{morse_property:open_set_in_open_set} to get $j\in \N$ such that $\u{M}{j}{y}\subset \u{M}{i}{\gamma (a)}$ for all $y\in \u{M}{j}{\gamma(a)}$. If $b$ is of finite type, set $k= j$. If $b$ is of infinite type, set $k = l(a)+1$.
\end{proof}

Let $M$ and $N$ be two Morse gauges with $M\leq N$. We have that both $\P{M}{A}\subset \P{N}{A}$ and $\M{M}{A}\subset\M{N}{A}$ and we have the same for $B$. Thus, the inclusion
\begin{align}
    i_{M ,N} : \co{M} \into \co{N}
\end{align}
is well-defined. More importantly, for any combinatorial $M$-Morse geodesic $a\in \co{M}$, we have $V_k^M(a) = V_k^N(a) \cap \co{M}$ and hence the inclusion $i_{M, N}$ is not only well-defined but also continuous. Therefore, we can define the combinatorial Morse boundary as follows:

\begin{definition}
The combinatorial Morse boundary of $\Gamma$ is defined as the direct limit 
\begin{align}
    \cm = \lim_{\xrightarrow[\mathcal{M}]{}}\co{M}.
\end{align}

\end{definition}

\begin{lemma}
The topology on $\co{M}$ induced by the fundamental system of neighbourhoods $\{V^{M}_{k}(a)\}_{k\in \N}$ coincides with the subspace topology of $\co{M}\subset \cm$.
\end{lemma}
\begin{proof}
It is enough to show that any closed $C\subset \co{M}$ satisfies that $C\cap\co{N}$ is closed for any Morse gauge $N$. In other words, we have to show that for any $a\in \co{N}\setminus C$ there exists a neighbourhood $V_k^N(a)$ such that $V_k^N(a)\cap C$ is empty. If there exists an integer $k$ such that $u_k(a)\notin \P{M}{A}\cup\P{M}{B}$ we are done, since in that case $V_k^N(a)\cap \co{M}$ is empty. If $a\in \co{M}$ we are done, as $V_{k+l}^N(a)\cap \co{M}\subset V_k^M(a)$ for some large enough $l$ and $C$ is closed in $\co{M}$. 

Thus, the only case left to consider is when $a$ is of finite type and $a\not\in \co{M}$. We may assume $\gamma(a)\subset A$. Define 
\begin{align*}
    Y = \{\gamma(b) \mid b\in C, v_{l(a)}(b) = v_{l(a)}(a)\}  \subset \M{M}{A}. 
\end{align*}
The set $Y$ is closed in $\M{M}{A}$ because $C$ is closed and thus $Y\cap\M{N}{A}$ is closed in $\M{N}{A}$. Therefore, there exists an integer $k_1$ such that $\u{N}{k_1}{\gamma(a)}\cap Y$ is empty. If we manage to find an integer $k_2$ such that $V_k^N(a)\cap \stab{M}\cap C$ is empty we can set $k = \max(k_1, k_2)$ and then $V_k^N(a)\cap C$ is empty. We show that such an integer $k_2$ exists by contradiction. Assume $V_k^N(a)\cap \stab{M}\cap C$ is nonempty for all positive integers $k$, then 
\begin{align*}
    Z_k = \{u_{l(a)+1}(b)\mid b\in C\cap V_k^N(a)\}
\end{align*}
is non-empty. Consider the sequence $(\gamma_i)_{\{i\in \N\}}$ such that $\gamma_i$ is an $M$-Morse realization of $x$ for some $x\in Z_i$. By Arzela-Ascoli, $(\gamma_i)_{\{i\in \N\}}$ contains a subsequence $(\gamma_{i_j})_{\{j\in \N\}}$ that converges uniformly on compact sets to some geodesic ray $\xi$. The proof of Lemma~2.10 of \cite{Cor16} shows that $\xi$ is $M$-Morse. Moreover, by the definition of $(\gamma_i)_{\{i\in \N\}}$, we have that $\xi \in \oh{N}{k}{\gamma(a)}$ for all $k$ and thus $\C\xi = \gamma(a)$. This is a contradiction to $a\not \in \co{M}$, which concludes the proof. 
\end{proof}

Below, we will show that the Morse boundary of $\Gamma$ is homeomorphic to its combinatorial Morse boundary. This will allow us to, in the next sections, work with the combinatorial Morse boundary, which is somewhat tamer than the Morse boundary itself.  

\subsection{Homeomorphism}
At the beginning of Section \ref{section2.1} we explained how, from a geodesic $\xi$ in $\Gamma$ starting at $e$ we can get a combinatorial Morse geodesic. We will call the combinatorial geodesic constructed this way $\Psi ( \xi)$. Note: If  $\xi$ is $M$-Morse, then $\Psi(\xi)$ is an $M_S$-Morse combinatorial geodesic. 

On the other hand, let  $a\in \cm$ be a combinatorial Morse geodesic. For all $1\leq i \leq l(a)$, let $p_i$ be a geodesic from $v_{i-1}(a)$ to $v_i(a)$. If $a$ is of infinite type we define 
\begin{align}
    \eta_a = p_1p_2\ldots 
\end{align}
If $a$ is of finite type we define 
\begin{align}
    \eta_a = p_1\ldots p_{l(a)} (v_{l(a)}\cdot \xi),
\end{align}
where $\xi$ is a realization of $\gamma(a)$. 
In both cases we define 
\begin{align}
    \Phi(a) = \C{\eta_a}.
\end{align} 

\begin{proposition}
The map 
\begin{align}
    \Phi : \cm \to \m  \Gamma
\end{align}
is well defined. Furthermore, if $a\in \co{M}$ is an $M$-Morse combinatorial geodesic $\Phi(a)$ is $M^\prime$-Morse, where $M^\prime$ depends only on $M$. 
\end{proposition}

\begin{proof}
Note that we first have to argue that $\eta_a$ is a geodesic. However, this follows from Lemma~\ref{combinatorial:geodesics}. Secondly, we show that when choosing appropriate geodesics $p_i$ and $\xi$, then $\eta_a$ is $M^\prime$-Morse where $M^\prime$ depends only on $M$ and lastly we show that $\C{\eta_a}$ does not depend on the choice of representatives $p_i$ and $\xi$.

For all $1\leq i\leq l(a)$, choose $p_i$ such that $p_i$ is $M$-Morse in $v_{i-1}(a)\cdot \Gamma_{G_{i\%2}}\subset  \Gamma$. Such a choice is possible since $u_i(a)\in \P{M}{G_{i\%2}}$. Furthermore, if $a$ is of finite type, choose $\xi$ to be an $M$-Morse realization of $\gamma(a)$ (again considered in $\Gamma_{G_{i\%2}}$). Lemma \ref{combinatorial:move_geodesics} shows that the $p_i$ and $\xi$ are $M^{\prime\prime}$-Morse in $\Gamma$, where $M^{\prime\prime}$ depends only on $M$. We claim that $\eta_a$ is $M^{\prime}$-Morse, where $M^{\prime}$ depends only on $M^{\prime\prime}$. 

Let $\gamma$ be a continuous $(\lambda, \eps)$-quasi-geodesic with endpoints on $\eta$. Assume $\gamma$ starts on $p_i$ for some $i$ and ends on some $p_j$ for some $j$ (the proof works analogously if one or more of them are on $v_{l(a)}(a)\cdot \xi$ instead). 
Assume that $i\leq j$, then by Lemma~\ref{combinatorial:geodesics} $\gamma$ goes through $v_i(a), \ldots , v_{j-2}(a)$ and $v_{j-1}(a)$. Thus we can split $\gamma$ into multiple parts, each having endpoints on an $M^{\prime\prime}$-Morse geodesic being a sub-geodesic of $\eta_a$. 
Therefore, each part of $\gamma$ is in an $M^{\prime\prime}(\lambda, \eps)$ neighbourhood of $\eta_a$. Using Remark~\ref{morse:continuity}, this shows that with this choice of $p_i$ and $\xi$, $\eta_a$ is an $M^\prime$-Morse geodesic, where $M^{\prime}$ depends only on $M$. 

Now we have to show that $\C{\eta_a}$ does not depend on the choice of $p_i$ and $\xi$. Or in other words, that if we choose different $\tilde{p}_i$ and $\tilde{\xi}$ to construct a geodesic $\tilde{\eta}_a$, then $\tilde{\eta}_a$ and $\eta_a$ have bounded Hausdorff distance. 
As $\C\xi = \C{\tilde{\xi}}$ it is clear that $v_{l(a)}\cdot \xi$ and $v_{l(a)}\cdot \tilde{\xi}$ have bounded Hausdorff distance. 
Lemma~2.5 from \cite{Charney_2014} shows that $p_i$ and $\tilde{p}_i$ have Hausdorff distance at most $C$, where $C$ depends only on $M^\prime$. Therefore, $\eta_a$ and $\tilde{\eta}_a$ have bounded Hausdorff distance, which concludes the proof.
\end{proof}

We can show that $\Phi$ is in fact a homeomorphism. 

\begin{lemma}
The map $\Phi$ is a bijection.
\end{lemma}

\begin{proof}
Let $\C{a} \in \m{\Gamma}$. We have $\Phi(\Psi(a)) = \C{a}$ and thus $\Phi$ is surjective. To prove injectivity, let $a\neq a^\prime\in \cm$ be combinatorial Morse geodesics, assume $l(a)\geq l(a^\prime)$. We want to show that the to $a$ and $a^\prime$ constructed geodesics $\eta_a$ and $\eta_{a^\prime}$ do not have bounded Hausdorff distance. If there exists an integer $i$ such that $v_i(a)\neq {v_i}(a^\prime)$ or if $l(a)>l(a^\prime)$, then for large enough $t$ any geodesic between ${\eta_a}(t)$ and ${\eta_{a^\prime}}$ goes through $v_i$. The distance $d(v_i(a), {\eta_a}(t))$ tends to infinity as $t$ goes to infinity. Therefore, the Hausdorff distance between ${\eta_a}$ and ${\eta_{a^\prime}}$ is not bounded and thus $\Phi(a)\neq\Phi(a^\prime)$. 
If $l(a) = l(a^\prime)$ and $v_i(a) = v_i(a^\prime)$ for all $1\leq i \leq l(a)$, we have $\C{\gamma(a)} \neq \C{\gamma(a^\prime)}$. 
Thus, $v_{l(a)}\cdot \xi $ and $v_{l(a^\prime)}\cdot \xi^\prime$ 
have unbounded Hausdorff distance and so do ${\eta_a}$ and ${\eta_a}^\prime$. 
\end{proof}

\begin{proposition}
The map $\Phi$ is a homeomorphism.
\end{proposition}

\begin{proof}
For a Morse gauge $M$, let $h(M)\geq M, M_S$ be a Morse gauge such that 
\begin{align}
\Phi(\co{M})\subset\M{h(M)}{\Gamma}\qquad \text{and}\qquad\Phi^{-1}(\M{M}{\Gamma}) \subset \co{h(M)}.
\end{align}
It is enough to show that the maps $\Phi_M : \co{M}\to\M{h(M)}{\Gamma}$ and $\Phi^{-1}_M:\M{M}{\Gamma} \to \co{h(M)}$ induced by $\Phi$ and $\Phi^{-1}$ respectively are continuous for all Morse gauges $M$.

First, we show that $\Phi_M$ is continuous. To do so, let $a\in \co{M}$ and let $k$ be a positive integer. We claim that
\begin{align}
    \Phi_M(V_{k+c}^M(a) ) \subset \u{h(M)}{k}{\Phi(a)},
\end{align}
where $c \geq 4\delta_{h(M)}+\delta_M+ 4 h(M)(3, 0)+1$. Let $b\in V_{k+c}^M(a)$. Assume $a$ is of finite type and let $D= d(e, v_{l(a)}(a))$. 
We have $d(\eta_b(k+c+D), \eta_a(k+c+D) ) < \delta_M$, therefore, Lemma~2.7 of \cite{Cor16} shows that for all $t< k+ D+2\delta_{h(M)}$
\begin{align}
    d(\eta_b(t), \eta_a(t))< \delta_{h(M)}.
\end{align} 
Therefore, $\eta_b\in \oh{h(M)}{k+4\delta_{h(M)}}{\eta_a}\subset \oh{h(M)}{k}{\Phi(a)}$. 

Assume $a$ is of infinite type. Note that $D = d(e, v_{k+c}(a))\geq k+c-1$, and that $d(\eta_a(D), \eta_b(D)) = 0$. Using Lemma~2.7 from \cite{Cor16} again we get that $\eta_b\in \oh{h(M)}{k+4\delta_{h(M)}}{\eta_a}\subset \oh{h(M)}{k}{\Phi(a)}$. This proves the claim and shows that $\Phi_M$ is continuous. 

Secondly, we prove that $\Phi^{-1}_M$ is continuous by showing that, for $a\in \M{M}{\Gamma}$,
\begin{align}
    \Phi^{-1}_M(\u{M}{l}{a}) \subset V^{h(M)}_{k}(\Phi^{-1}(a)),
\end{align}
where we will define $l$ later. Assume $\Phi^{-1}(a)$ is of finite type and let $b\in\u{M}{l}{a}$. We use the following notation: \begin{itemize}
    \item $\Phi^{-1}(a) = \hat{a}$ and $\Phi^{-1}(b) = \hat{b}$
    \item $v_{l(\hat{a})}(\hat{a}) = w$
    \item $D= d(e, w)$ and $l = k+ D+ 4\delta_{h(M)}+\delta_M$
    \item $\xi_a$ and $\xi_b$ are $M$-Morse realizations of $a$ and $b$ respectively
    \item $w^{-1}\cdot \xi_a[D, \infty) =  \zeta_a$ and $w^{-1}\cdot \xi_b[D, \infty) = \zeta_b$.
\end{itemize}
Any realization of $b$ goes through $w$, as $\xi_a[D, \infty)$ is a subset of $\Gamma[C_w]$. Therefore $\zeta_a$ and $\zeta_b$ are $M$-Morse geodesics starting at $e$. More importantly, $\zeta_a\subset \Gamma_A$ or $\zeta_a\subset \Gamma_B$. We may assume that $\zeta_a\subset \Gamma_A$ and show in this case that $\zeta_b[0, k + 4\delta_{h(M)}]\subset \Gamma_A$. Together with $\delta_M\leq\delta_{h(M)}$ this then implies that $\hat{b}\in V^{h(M)}_{k}(\hat{a})$, which is what we needed to show. 
Assume $\zeta_b[0, k+4\delta_{h(M)}]\not\subset \Gamma_A$, let $t\geq 0$ be the largest number such that $\zeta_b(t)\in \Gamma_A$. Then $\zeta_b[t, \infty)\subset \Gamma[C_{\zeta_b(t)}]$ while $\zeta_a\subset\Gamma[\bar{C}_{\zeta_b(t)}]$. 
Therefore, $d(\zeta_b(t+\delta_M), \zeta_a(t+\delta_M))>\delta_M$ which is a contradiction to $b\in \u{M}{l}{a}$ if $t< k + 4\delta_{h(M)}$. Hence $\zeta_b[0, k+ 4\delta_{h(M)}]\subset \Gamma_A$ which concludes this case. 

At last, assume $\Phi^{-1}(a)$ is of infinite type and let $b\in\u{M}{l}{a}$. Set $\Phi^{-1}(a) = \hat{a}$ and $\Phi^{-1}(b) = \hat{b}$. Furthermore, define
\begin{itemize}
    \item $w = v_k(\hat{a})$
    \item $D = d(e, w)$ and $l = D + \delta_{M}$
    \item $\xi_a$ and $\xi_b$ are $M$-Morse realizations of $a$ and $b$ respectively
\end{itemize}
We have $\xi_a [D, \infty) \subset \Gamma[C_w]$, so $\xi_b\cap\Gamma[C_w]\neq \emptyset$, which implies $\hat{b}\in V_k^{h(M)}(\hat{a})$.
\end{proof}

\section{Free product}\label{chapter:map}
In this section, we prove Theorem \ref{freeproduct_theorem}. To do so, we use the following set up: Let $p_A : \m A_1 \to \m A_2$ and $p_B : \m B_1\to \m B_2$ be homeomorphisms. We denote by $p$ the map from $\m A_1 \cup\m B_1$ to $\m A_2 \cup \m B_2$ their union induces. Recall that in Section \ref{section:homeo_of_morse_boundaries} we have seen that there exists a map $h: \mathcal M\to \mathcal M$ that describes the Morseness of $p(z)$ for a Morse direction $z\in \m A_1$. More precisely, there exists a map $h: \mathcal M \to \mathcal M$ that satisfies \eqref{eq:homoe_of_morse_strata}.

\subsection{Outline of the proof}\label{section:proof_outline}
We start with describing the overall structure of the proof. After that, we describe the various steps in more details. \\

\paragraph{Coarse structure of the proof} First, we use $p$ to construct bijections $\hat{p}_A : A_1 \to A_2$ and $\hat{p}_B : B_1 \to B_2$. 
We make sure that the bijection $\hat{p}_A$ (and analogously its inverse $\hat{p}_A^{-1}$,  $\hat{p}_B$ and $\hat{p}^{-1}_B$) we construct respects the continuity of $p$ and the Morseness of elements in the following way. 

\begin{enumerate}[label = \Roman*)]
    \item If $x\in A_1$ is $M$-Morse, then $\hat{p}_A(x)$ is $M'$-Morse, where $M'$ depends on $M$ but not on $x$. More precisely, Lemma \ref{local_map:morse_bound} is satisfied. \label{cond:map_morse}
    \item Every filled neighbourhood $\U{M'}{l}{p(z)}$ of a Morse direction $p(z)$ contains the image (under $\hat{p}_A\cup p$) of a filled neighbourhood $\U{M}{k}{z}$. Recall that a filled neighbourhood is a set that contains both points in the boundary and elements of the group. More precisely, Proposition \ref{local_map:continuity} is satisfied. \label{cond:map_cont}
\end{enumerate}

We also ensure that $\hat{p}_A(e)=e$ and $\hat{p}_B(e)=e$. Condition \ref{cond:map_morse} implies that $\hat{p}_A$, $\hat{p}_B$ and $p$ induce a bijection $\bar{p} : \delta_* (A_1*B_1) \to \delta_*(A_2*B_2)$ on the combinatorial Morse boundary. Condition \ref{cond:map_cont} can be used to show that $\bar{p}$ (and analogously its inverse) is continuous. Since the combinatorial Morse boundary is homeomorphic to the Morse boundary,  it follows that $\m(A_1*B_1)$ and $\m(A_2*B_2)$ are homeomorphic.\\

\paragraph{Constructing the bijection $\hat{p}_A$} In the case that $\m A_1$ is empty, any bijection $\hat{p}_A : A_1\to A_2$ satisfies Conditions \ref{cond:map_morse} and \ref{cond:map_cont}. In that case, we define $\hat{p}_A$ as an arbitrary bijection satisfying $\hat{p}_A(e) = e$. Hence for the rest of the construction, we assume that $\abs{\m A_1}\geq 2$. In this case, we construct $\hat{p}_A$ iteratively. First, we define $\hat{p}_A(e) = e$. Next, we arbitrarily enumerate the elements of $A_1$ and $A_2$. This enumeration determines in which order we define $\hat{p}_A$. Namely, we will alternate the following two steps:
\begin{itemize}
    \item \textbf{Step 1:} Take the smallest index (with respect to the chosen enumeration) element $x\in A_1$ for which $\hat{p}_A$ is not yet defined. Define $\hat{p}_A(x)$.
    \item \textbf{Step 2:} Take the smallest index (with respect to the chosen enumeration) element $y\in A_2$ for which $\hat{p}_A^{-1}$ is not yet defined. Define $\hat{p}_A(y)^{-1}$.
\end{itemize}
The alternating between the two steps ensures that firstly, $\hat{p}_A$ is indeed a bijection. If we would not alternate, we could only guarantee that $\hat{p}_A$ is injective, but not necessarily surjective. Secondly, it allows us to have a certain symmetry between $\hat{p}_A$ and $\hat{p}_A^{-1}$ (and thus later $\bar{p}$ and $\bar{p}^{-1})$. For example we then only need to prove Condition \ref{cond:map_cont} for $\hat{p}_A$ and get the analogous statement about $\hat{p}_A^{-1}$ for free.

Now we know in which order we want to define $\hat{p}_A(x)$, but we still have to describe how we define it for a fixed $x\in A_1$. So next, we explain how to construct $\hat{p}_A(x)$ in Step 1. Recall that the construction in Section \ref{section:morse_geodesic_lines} defines for every group element $x\in A_1$ a geodesic ray $\lambda_x$ corresponding to $x$. Lemma \ref{map:morseline:porperties} states that $\lambda_x$ captures the Morseness properties of $x$ and passes close by $x$. Thus, points on the realization of $p(\lambda_x)$ are candidates for the choice of $\hat{p}_A(x)$. This guarantees that $\hat{p}_A(x)$ satisfies Condition \ref{cond:map_morse}. However, we need that $\hat{p}_A^{-1}$ also satisfies Condition \ref{cond:map_morse}. In other words, choosing $\hat{p}_A(x)$ on $p(\lambda_x)$ ensures that $\hat{p}_A(x)$ is nice enough compared to $x$ but we also need to make sure that $\hat{p}_A(x)$ is not too nice compared to $x$. Not every point on $p(\lambda_x)$ can be too nice and points closer to the basepoint $e$ are nicer than points further away. Thus the idea is that if we choose $\hat{p}_A(x)$ being far enough along $p(\lambda_x)$ Condition \ref{cond:map_morse} will be satisfied for $\hat{p}_A^{-1}$. 

It is not trivial to define what far along enough should mean. Most of the complexity comes from the fact that we are dealing with a potentially non-$\sigma$-compact Morse boundary and thus we cannot talk about ``minimal'' Morse gauges. However, we can use that $p^{-1}$ is continuous as well and hence use Lemma \ref{homeo_properties:back_match_lambda_x} to determine what far along enough means.

At least for one of the points $y$ far enough along $p(\lambda_x)$, $\hat{p}_A(y)^{-1}$ has not yet been defined (there are infinitely many points $y$ far enough along $p(\lambda_x)$ but only for finitely many of them $\hat{p}_A(y)^{-1}$ has been defined). We can choose any such $y$ and define $\hat{p}_A(x) = y$.

The construction of $\hat{p}_A^{-1}(y)$ in Step 2 is the same (with the roles of $A_1$ and $A_2$, resp $p$ and $p^{-1}$ reversed).\\

\paragraph{Proving Condition \ref{cond:map_cont}} Proving Condition \ref{cond:map_cont} requires a series of technical lemmas. The idea is to use that $x$ and $\lambda_x$ are close, Remark \ref{lemma1.16_direct} quantifies this. Similarly, $\hat{p}_A(x)$ and $p(\lambda_x)$ are close (in fact $\hat{p}_A(x)$ lies on $p(\lambda_x)$). Since we can control what happens with neighbourhoods under $p$, this closeness allows us to control what happens with filled neighbourhoods under $\hat{p}_A\cup p$.

Now we are ready to prove Theorem \ref{freeproduct_theorem} in full detail.

\subsection{Bijections on the factors}\label{section:bijections_on_factors}
If $\m A_1= \emptyset$ we let $\hat{p}$ be any bijection from $A_1$ to $A_2$ that satisfies $\hat{p}(e)=e$. Since Lemma \ref{lemma:not_one} implies that $\abs{\m A_1}\neq 1$ we may assume that $\abs{\m A_1}\geq 2$ for the rest of the section. In this case, we construct the bijection $\hat{p} : A_1 \to A_2$ as follows: Set $\hat{p}(e)=e$ and enumerate the elements of $A_1$ and $A_2$. Then, repeat the following two steps: 
\begin{itemize}
    \item \textbf{Step 1:} Denote by $x$ the lowest index element of $A_1$ which has not yet been matched. Assume $\lambda_x$ is $M$-Morse for some Morse gauge $M$ and let $\eta$ be an $h(M)$-Morse realization of $p(\lambda_x)$. Let $T_1$ be as in Lemma~\ref{map:morseline:porperties}~\ref{map:morseline:properties:3} and define $M^\prime = h(h(M)_\lambda)$. Set $T= \max\{T_1 + 4\delta_{M^\prime}, 12\delta_{M^\prime}\}$. Let $y= \eta(i)$ for some $i\in \N$ such that $\eta(i)$ is still unmatched and satisfies \eqref{cond:far_along} of Lemma \ref{homeo_properties:back_match_lambda_x}, namely 
    \begin{align}\label{map:construction:back_mapping}
        p^{-1}(\lambda_{\eta(i)}) \in \u{M^\prime}{T}{\C{\lambda_x}}.
    \end{align}
    Match $x$ to $y$, that is, set $\hat{p}(x) = y$. Figure \ref{picture:construction_of_map} depicts this construction.
     We say that $x$ is the \emph{initiator} and $y$ is the \emph{target}.
    \item \textbf{Step 2:} Denote by $y$ the lowest index element of $A_2$ which has not yet been matched. Repeat the process from Step 1, with the roles of $A_1$ and $A_2$ reversed, to construct a point $x\in A_1$. Set $\hat{p}(x) = y$. Note that in this case, $y$ is called the initiator and $x$ the target. 
\end{itemize}
\begin{figure}\centering
\includegraphics[width= \linewidth]{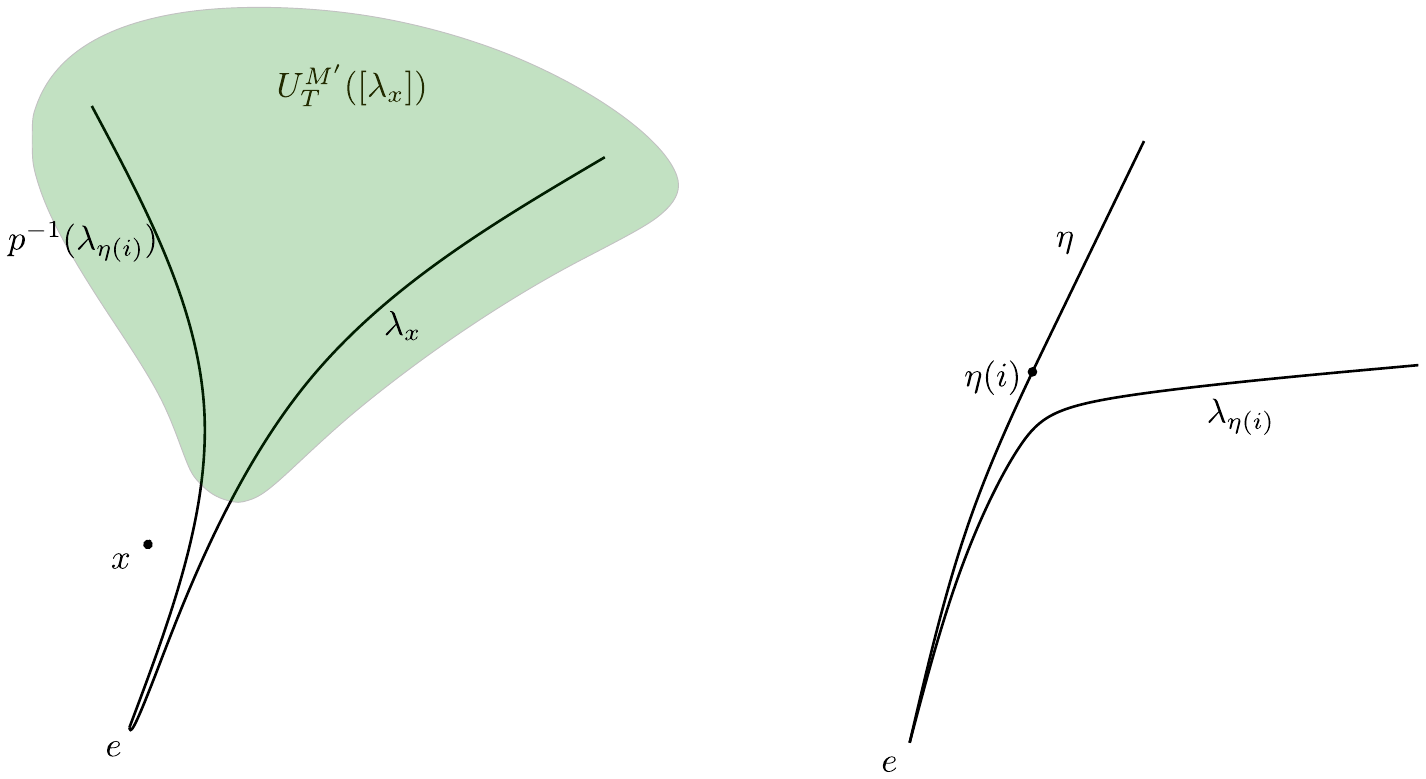}
\caption{Construction of $\hat{p}(x)$.}
\label{picture:construction_of_map}
\end{figure}

Observe that Lemma \ref{homeo_properties:back_match_lambda_x} implies all but finitely many positive integers satisfy \eqref{map:construction:back_mapping}. Since also all but finitely many elements of $A_2$ are unmatched, an element $\eta(i)$ with the desired properties indeed exists and hence $\hat{p}$ is well-defined. 

When we work with some $x\in A_1$ in the matching process, we take any Morse gauge $M$ such that $\lambda_x$ is $M$-Morse and use it to construct the target $y$. We make no assumptions that the Morse gauge $M$ is in any way ``optimal". However, it is important for the well-definedness of later constructions, that if $x$ is $N$-Morse, then $y$ is $N^\prime$-Morse where $N^\prime$ depends only on $N, h$ and $M_0$. This is why we posed the technical condition \eqref{map:construction:back_mapping}. It allows us to prove the following lemma:  

\begin{lemma}\label{local_map:morse_bound}
If $x$ is $N$-Morse, then $\hat{p}(x)$ is $N^\prime$-Morse, where $N^\prime$ depends only on $N$, $M_0$ and $h$. Similarly, if $\hat{p}(x)$ is $N$-Morse, then $x$ is $N^\prime$-Morse, where $N^\prime$ depends only on $N, M_0$ and $h$.
\end{lemma}

\begin{proof}
In this proof we use the notation from Step 1 of the construction of $\hat{p}$. Assume $x\in \P{N}{A_1}$ is an initiator. Then, $\lambda_x$ is $N_\lambda$-Morse and $p(\lambda_x)$ is $h(N_\lambda)$-Morse. Lemma \ref{morse:bound_on_realization} then implies that $\eta$ is $N^\prime = h(N_\lambda)_\mathcal{C}$-Morse. As $\hat{p}(x)$ lies on $\eta$ it is $N_S^\prime$-Morse. Note that $N_S^\prime$ depends only on $N$, $M_0$ and $h$. Therefore, part one of the lemma is true. It also shows that part two is true for any $x\in A_1$ that is a target. 

To prove the other direction, assume that $x\in A_1$ is an initiator and that $\hat{p}(x) = \eta(i)$ is $N$-Morse. Observe that $\lambda_{\eta(i)}$ is $N_\lambda$-Morse. Let $\xi\in \oh{M^\prime}{T}{\C{\lambda_x}}$ be an $M^\prime$-Morse realization of $p^{-1}(\lambda_{\eta(i)})$. 

Note that $\C\xi$ is $h(N_\lambda)$-Morse and thus $\xi$ is not only $M^\prime$-Morse but also $N_1$-Morse, where $N_1 =h(N_\lambda)_\mathcal{C}$.
Therefore, using Cor 2.5 of \cite{Cor16}, we get $d(\xi(t), \lambda_x(t))< \delta_{N_1} $ for all $t\in [0, T - 4\delta_{M^\prime}]$. Since $T-4\delta_{M^\prime}\geq T_1$, the Hausdorff distance between $\lambda_x[0, T_1]$ and $\xi[0, T_1]$ is bounded by $\delta_{N_1}$. Hence Lemma~2.5~i) of \cite{Charney_2014} implies that $\lambda_x [0, T]$ is $N_2$-Morse, where $N_2$ depends only on $N_1$. Lemma~\ref{morse_property:concatenation_of_geodesic} together with the fact that $\lambda_x[T_1, \infty)$ is $M_1$-Morse implies that $\lambda_x$ is $N_3$-Morse, where $N_3$ depends only on $M_1$ and $N_2$. By Lemma~\ref{map:morseline:porperties}~\ref{map:morseline:properties:4}, $x$ is $N_4$-Morse, where $N_4$ depends only on $N_3$ and $M_0$. One can check step by step that $N_4$ only depends on $N$, $M_0$ and $h$, which is what we wanted to show. It also shows that part one is true for the case that $x$ is a target.  
\end{proof}

In the construction, we did not match an element $x\in A_1$ to another element $y\in A_2$ directly but first matched $x$ to a geodesic, used $p$ to get a geodesic in $A_2$ and then got a point $y\in A_2$. Those intermediary geodesics are the connection between $p$ and $\hat{p}$ and are important in proving that $\hat{p}$ respects the continuity of $p$, in the sense that Proposition \ref{local_map:continuity} is satisfied. This motivates the following definition: 

For any $x\in A_i$ for $i\in \{1, 2\}$ which is an initiator define $g(x)$, \emph{the geodesic matched to} $x$, as $\C{\lambda_x}$. For any $y\in A_2$ which is a target, define $g(y)$ as $p(g(\hat{p}^{-1}(y)))$. In other words, if $y$ is a target and $y = \hat{p}(x)$, then $g(y) = p(\lambda_x)$ and $y$ lies on $g(y)$. Similarly, for any $x\in A_1$ which is a target, define $g(x)$ as $p^{-1}(g(\hat{p}(x)))$. 

\begin{lemma}
Let $x\in A_i$ for $i\in \{1, 2\}$ be $M$-Morse. Then, $g(x)$ is $M^\prime$-Morse, where $M^\prime$ depends only on $M$, $M_0$ and $h$.
\end{lemma}
\begin{proof}
If $x$ is an initiator, then the statement holds due to Lemma~\ref{map:morseline:porperties}~\ref{map:morseline:properties:1}. If 
$x\in A_1$ is a target, then Lemma~\ref{local_map:morse_bound} implies $\hat{p}(x)$ is $M_1$-Morse, where $M_1$ depends only on $M$, $M_0$ and $h$. Observe that then $\hat{p}(x)$ is an initiator, and thus $g(\hat{p}(x))$ is $M_2$-Morse, where $M_2$ depends only on $M_1$ and $M_0$. As $g(x) = p^{-1}(g(\hat{p}(x)))$, it is $h(M_2)$-Morse. Analogously we can prove the statement for a target $x\in A_2$. 
\end{proof}

Recall that $M_\lambda$ was defined in Lemma \ref{map:morseline:porperties}. The following notation will be useful: Let $M$ be a Morse-gauge, we define $M_p\geq M, M_\lambda, h(M)$ as a Morse gauge such that:
\begin{itemize}
    \item For all $x\in A_1$ that are $M$-Morse  $\hat{p}(x), g(x)$ and $\lambda_x$ are $M_p$-Morse.
    \item For all $y\in A_2$ that are $M$-Morse  $\hat{p}^{-1}(y), g(y)$ and $\lambda_y$ are $M_p$-Morse.
\end{itemize}

From now on, by abuse of notation, we denote by $p$ the map $A_1\cup \m{A_1}\to A_2\cup \m{A_2}$ induced by $\hat{p}\cup p$. Recall that for a Morse direction $z\in \M{M}{A_1}$, the filled neighbourhood $\U{M}{l}{z}$ is the $M$-Morse $l$-neighbourhood of $z$ which also includes elements of $A_1$. Our last goal of this subsection is to prove the following proposition, which states that every filled neighbourhood of a Morse direction $p(z)$ contains the image of a filled neighbourhood of $z$.

\begin{proposition}\label{local_map:continuity}
For all positive integers $l$ and $M$-Morse directions $z\in \M{M}{A_1}$ there exists an integer $k$ such that $p(\U{M}{k}{z})\subset \U{M_{p^2}}{l}{p(z)}$. 
\end{proposition}

Note that $M_{p^2}$ denotes $(M_p)_p$. Before we start with the proof of Proposition \ref{local_map:continuity} we want to highlight one of its consequences.

\begin{corollary}\label{cor:homeo_induced_by_bijection}
Let $G, H$ be finitely generated groups with homeomorphic Morse boundaries. Then, for every homeomorphism $p: \m G \to \m H$ there exists a bijection $\hat{p} : G\to H$ which induces $p$. 
\end{corollary}

\begin{proof}
If $\m G$ is empty, the statement follows trivially. So, we assume that $\m G$ is non-empty. Let $p : \m G\to \m H$ be a homeomorphism. We use the construction from above to get a bijection $\hat{p} : G\to H$ and want to prove that $\hat{p}$ induces $p$. 

Let $M$ be a Morse gauge and let $(x_n)_n$ be a sequence of points in $G$ that converges to some Morse direction $z\in \m G$. We need to show that the image of $(x_n)_n$ converges to $p(z)$. Let $l$ be any integer. Proposition \ref{local_map:continuity} implies that there exists an integer $k$ such that  $(p\cup \hat{p})(\U{M}{k}{z})\subset \U{M_{p^2}}{l}{p(z)}$. Since the sequence $(x_n)_n$ converges to $z$, all but finitely many elements of $(x_n)_n$ are contained in $(\U{M}{k}{z})$. Hence, all but finitely many elements of $(\hat{p}(x_n))_n$ are contained in $\U{M_{p^2}}{l}{p(z)}$, which implies that the sequence $(\hat{p}(x_n))_n$ indeed converges to $p(z)$. Since $\hat{p}$ and $\hat{p}^{-1}$ are constructed analogously, the argument can be repeated for $\hat{p}^{-1}$. Thus, the bijection $\hat{p}$ indeed induces the homeomorphism $p$.

\end{proof}
To prove Proposition \ref{local_map:continuity} we need the following two technical lemmas.

\begin{lemma}\label{local_map:g(x)_contained_in_open_set}
Let $k$ be a positive integer. For all $x\in \P{M}{A_1}$ satisfying $d(e, x)\geq C_M(k+4\delta_{M_p}$) we have
\[
g(x)\in \u{M_p}{k}{x} \qquad\text{and}\qquad x\in \U{M_p}{k}{g(x)}.
\]
\end{lemma}
\begin{proof}

Assume $x$ is an initiator. Then, $g(x) = \C{\lambda_x}$. Let $\eta$ be an $M$-Morse realization of $x$. Remark~\ref{lemma1.16_direct} shows that $\eta\in\Oh{M_\lambda}{k+ 4\delta_{M_p}}{\lambda_x}$. This implies that $\eta \in\Oh{M_p}{k}{\C{\lambda_x}}$ and that $\lambda_x\in \oh{M_\lambda}{k+ 4\delta_{M_p}}{\eta}$. Consequently, $\lambda_x\in\oh{M_p}{k}{x}$. 

Assume $x$ is a target. Then, there exists an $M_p$-Morse realization $\xi$ of $g(x)$ that goes through $x$. Note that $\eta = \xi[0, d(e, x)]$ is an $M_p$-Morse realization of $x$ satisfying $\eta\in\Oh{M_p}{d(e, x)}{\xi}$ and $\xi\in\oh{M_p}{d(e, x)}{\eta}$. Thus $\eta\in\Oh{M_p}{k}{g(x)}$ and $\xi\in\oh{M_p}{k}{x}$.
\end{proof}

\begin{lemma}\label{local_map:match}
Let $l$ be a positive integer and let $z\in \M{M}{A_1}$ be an $M$-Morse direction. There exists an integer $k$ such that
\begin{enumerate}[label = \roman*)]
    \item for all $x\in\U{M}{k}{z}\cap A_1$ we have $g(x)\in \U{M_p}{l}{z}$.\label{local_map:match_g(x)_if_x}
    \item all $x\in \P{M}{A_1}$ with $d(e, x)\geq k$ and $g(x)\in \U{M_p}{k}{z}$ satisfy $x\in \U{M_p}{l}{z}$.\label{local_map:match_x_if_g(x)}
\end{enumerate}
\end{lemma}
\begin{proof}
Lemma~\ref{morse_property:open_set_in_open_set} implies that there exists an integer $k^\prime$ such that for all $y\in\U{M_p}{k^\prime}{z}$ we have 
\begin{align}
\u{M_p}{k^\prime}{y}\subset \U{M_p}{l}{z}. \label{lemma3.8_eq1}
\end{align}
Set $k = C_M(k^\prime +2\delta_{M_p})$. Suppose $x\in\U{M}{k}{z}\cap A_1$ and let $\xi$ be an $M$-Morse realization of $z$. We have $x\in \U{M}{k}{\xi}\subset\U{M}{k^\prime + 2\delta_{M_p}}{\xi}\subset \U{M_p}{k^\prime + 2\delta_{M_p}}{\xi}\subset\U{M_p}{k^\prime}{z}$. Due to Lemma~\ref{local_map:g(x)_contained_in_open_set} we have $g(x)\in \U{M_p}{k^\prime}{x}$ and thus \eqref{lemma3.8_eq1} concludes the proof of i). 

Suppose $g(x)\in \U{M_p}{k^\prime}{z}$ and $d(e, x)\geq k$. Lemma \ref{local_map:g(x)_contained_in_open_set} shows that $x\in \U{M_p}{k^\prime}{g(x)}$ and thus $\eqref{lemma3.8_eq1}$ also concludes the proof of ii).
\end{proof}

\begin{proof}[Proof of Proposition \ref{local_map:continuity}] Let $z\in \M{M}{A_1}$ be an $M$-Morse direction and $l$ a positive integer.
First, use Lemma~\ref{local_map:match}~\ref{local_map:match_x_if_g(x)} to get an integer $k_1$ such that all $y\in \P{M_p}{A_1}$ with $d(e, y)\geq k_1 $ and $g(y) \in \U{M_{p^2}}{k_1}{p(z)}$ satisfy $y\in\U{M_{p^2}}{l}{p(z)}$.
Let $k_2$ be such that $p(\u{M_p}{k_2}{z})\subset \u{M_{p^2}}{k_1}{p(z)}$. Then, use Lemma~\ref{local_map:match}~\ref{local_map:match_g(x)_if_x} to get an integer $k_3$ such that for all $x\in \U{M}{k_3}{z}\cap A_1$ we have that $g(x)\in\u{M_p}{k_2}{z}$. 
Finally, let $k\geq k_3$ such that all $x\in A_1$ with $d(e, x)\geq k$ satisfy $d(e, p(x))\geq  k_1$ and such that $p(\u{M}{k}{z})\subset\u{M_{p^2}}{l}{p(z)}$.

This choice of $k$ works: Let $x\in \U{M}{k}{z}$ we want to show that $p(x) \in \U{M_{p^2}}{l}{p(z)}$. If $x\in \m A_1$, the statement follows directly. Therefore, we assume $x\in A_1$. By choice of $k_3$ we have $g(x)\in \u{M_p}{k_2}{z}$ and by choice of $k_2$ we have $p(g(x)) \in \u{M_{p^2}}{k_1}{p(z)}$. Note that, by the choice of $k$, $d(e, p(x))\geq k_1$ and thus $p(x)\in \U{M_{p^2}}{l}{p(z)}$.
\end{proof}

Proposition~\ref{local_map:continuity} also holds if $\m A_1 = \emptyset$, as any statement about elements of the empty set holds. Furthermore, if $\m A_1 = \emptyset$, then $\M{M}{A}$ is finite for any Morse gauge $M$. Therefore, Lemma \ref{local_map:morse_bound} also holds in that case.  Apart from $\hat{p}(e)=e$, those are the only properties about $\hat{p}$ that we will use later on. Thus, from now on, we do not have to make a case distinction whether $\m A_1$ is empty or not. 

\subsection{Proof of Theorem \ref{freeproduct_theorem}}
We define the map $\bar{p} : \cm_1 \to \cm_2$ via 
\begin{align}
    \bar{p}(u_1, \ldots, u_n; x) &= (\hat{p}(u_1), \ldots, \hat{p}(u_n); p(x))
\end{align}
and
\begin{align}
    \bar{p}((u_i)_{i\in \N}) &= (\hat{p}(u_i))_{i\in\N}.
\end{align}

The map $\bar{p}$ is well defined as the image of any combinatorial $M$-Morse geodesic is a combinatorial $M_p$-Morse geodesic (and as $\hat{p}(e)= e$). The map is a bijection as $p$ and $\hat{p}$ are a bijection. We are now ready to prove Theorem \ref{freeproduct_theorem}

\begin{proof}[Proof of Theorem \ref{freeproduct_theorem}]
It is enough to show that $\overline{p}$ is continuous since then, by symmetry, $\bar{p}^{-1}: \cm_2 \to \cm_1$ is continuous as well. 

To do so, we will show that for any Morse gauge $M$, $x\in\co{M}_1$ and $l\in\N$ there exists an integer $k\in \N$ such that $\bar{p}(\vh{M}{k}{x})\subset\vh{M_{p^2}}{l}{\bar{p}(x)}$.

Case 1: $x$ is in $\T{M}_1$: In this case, we can set $k=l$. Indeed: Let $x^\prime \in  \vh{M}{l}{x}$, then $x^\prime$ is an $M$-Morse combinatorial geodesic with $l(x^\prime)\geq l+1$ and $u_i(x^\prime) = u_i(x)$ for all $1\leq i \leq l$. Thus $\bar{p}(x^\prime)$ is an $M_p$-Morse combinatorial geodesic with $l(\bar{p}(x^\prime))\geq l+1$ and $u_i(\bar{p}(x^\prime)) = u_i(\bar{p}(x))$. Hence $\bar{p}(x^\prime)\in \vh{M_{p^2}}{l}{\bar{p}(x)}$.

Case 2: $x$ is in $\stab{M}_1$: We use Proposition~\ref{local_map:continuity} to get an integer $k$ such that $p(\U{M}{k}{\gamma(x)})\subset \U{M_{p^2}}{l}{\gamma(\bar{p}(x))}$. Let $x^\prime \in \vh{M}{k}{x}$. Observe that $\bar{p}(x^\prime)$ is $M_{p^2}$ Morse. Furthermore:
\begin{enumerate}[label= \roman*)]
    \item $l(x^\prime)\geq l(x)$ and thus $l(\bar{p}(x^\prime))\geq l(\bar{p}(x))$.
    \item $u_i(x^\prime) = u_i(x)$ for all $1\leq i \leq l(x)$ and thus $u_i(\bar{p}(x^\prime)) = u_i(\bar{p}(x))$ for all $1\leq i \leq l(\bar{p}(x))$.
    \item if $l(x^\prime) > l(x)$, then $u_{l(x)+1}(x^\prime) \in \U{M}{k}{\gamma(x)}$ $l(x^\prime) > l(x)$ and therefore $u_{l(x)+1}(\bar{p}(x^\prime))\in \U{M_p}{l}{\gamma(\bar{p}(x))}$.
    \item if $l(x^\prime) = l(x)$, then $\gamma(x^\prime) \in \U{M}{k}{\gamma(x)}$ and thus $\gamma(\bar{p}(x^\prime))\in \U{M}{l}{\gamma(\bar{p}(x))}$.
\end{enumerate}
Therefore, $\bar{p}(x^\prime)\in \vh{M_{p^2}}{l}{\bar{p}(x)}$, which concludes the proof. 
\end{proof}

\section{Proof of Theorem \ref{thm:graph_of_groups}}\label{chapter:graph_of_groups}
The proof of Theorem \ref{thm:graph_of_groups} works the same way as the proof of Theorem 1.1 of \cite{martin2013hyperbolic} and follows from the Theorem of P. Papasoglu and K. Whyte stated below (see \cite{papasoglu2002quasi}). We will state a slight adaptation of the version used in \cite{martin2013hyperbolic}. 
\begin{theorem}[Papasoglu-Whyte]
Let $G, H$ be finitely generated groups with infinitely many ends and let $\mathcal{G}, \mathcal{H}$ be their graph of groups decompositions with all edge groups finite. If $\mathcal{G}$ and $\mathcal{H}$ have the same set of quasi-isometry types of infinite vertex groups (without multiplicities) then $G$ and $H$ are quasi-isometric. In particular, if $G_1, \ldots, G_m$ represent all quasi-isometry types of the infinite vertex groups of $\mathcal{G}$ which are not virtually cyclic, 
\begin{itemize}
    \item if $m=0$, then $G$ is quasi-isometric to the free group $F_2$;
    \item if $m=1$, then $G$ is quasi-isometric to $G_1\ast G_1$;
    \item if $m>1$, then $G$ is quasi-isometric to $G_1\ast \cdots \ast G_m$.
\end{itemize}
\end{theorem}

\begin{proof}[Proof of Theorem \ref{thm:graph_of_groups}]
The Morse boundary is a quasi-isometry invariant (see \cite{Cor16}). If $m=0$ for the group $G$, then $m=0$ for the group $H$, as a group is quasi-isometric to a virtually cyclic group if and only if it is virtually cyclic itself. In the other two cases, $G$ is quasi-isometric to $G_1\ast\cdots \ast G_m$ and $H$ is quasi-isometric to $H_1\ast \cdots \ast H_m$ (here we set $G_2 = G_1$ if $m=1$ and do the same for $H$). We can assume that the groups $H_1, \ldots , H_m$ are ordered in a way such that $H_i$ is quasi-isometric to $G_i$. Therefore $H_i$ and $G_i$ have homeomorphic Morse boundaries. Using induction and Theorem~\ref{freeproduct_theorem} as base-case and step, we get that $\m{(G_1\ast \cdots \ast G_m)}\cong \m{(H_1\ast \cdots \ast H_m)}$. The theorem follows from the Morse boundary being a quasi-isometry invariant. 
\end{proof}

\bibliographystyle{alpha}
\bibliography{sources}
\end{document}